\documentclass[11pt,leqno,a4paper]{amsart}

\usepackage[margin=25mm]{geometry}

\usepackage[title]{appendix}
\usepackage{amssymb}
\usepackage{amsmath}
\usepackage{amsthm}
\usepackage{tikz}
\usepackage{tikz-cd}
\usepackage{mathtools}
\usepackage{framed}
\usepackage{comment}
\usepackage{booktabs}
\definecolor{darkblue}{rgb}{0,0,0.6}
\usepackage[breaklinks, ocgcolorlinks,colorlinks=true, citecolor=darkblue, filecolor=darkblue, linkcolor=darkblue, urlcolor=darkblue]{hyperref}
\usepackage[capitalize]{cleveref}

\newcommand{\Q}{\mathbb{Q}}
\newcommand{\R}{\mathbb{R}}

\newcommand{\Z}{\mathbb{Z}}
\renewcommand{\H}{\mathbb{H}}
\newcommand{\F}{\mathbb{F}}

\renewcommand{\mod}{\,\,\text{\normalfont mod}\,}

\DeclareMathOperator{\Aut}{Aut}

\DeclareMathOperator{\id}{id}

\DeclareMathOperator{\image}{image}

\DeclareMathOperator{\IM}{im}

\DeclareMathOperator{\Ker}{Ker}
\DeclareMathOperator{\coker}{coker}

\DeclareMathOperator{\GL}{GL}

\DeclareMathOperator{\ord}{ord}

\DeclareMathOperator{\Inn}{Inn}
\DeclareMathOperator{\Out}{Out}
\DeclareMathOperator{\SF}{SF}

\DeclareMathOperator{\HT}{HT}

\DeclareMathOperator{\FF}{F}
\DeclareMathOperator{\LL}{\Lambda}

\DeclareMathOperator{\per}{period}
\DeclareMathOperator{\lcm}{lcm}
\DeclareMathOperator{\colim}{colim}
\DeclareMathOperator{\hAut}{hAut}

\newtheorem{thm}{Theorem}[section]

\newtheorem*{thm*}{Theorem}
\newtheorem{prop}[thm]{Proposition}
\newtheorem*{prop*}{Proposition}
\newtheorem{lemma}[thm]{Lemma}
\newtheorem{corollary}[thm]{Corollary}
\newtheorem*{corollary*}{Corollary}

\newtheorem{question}[thm]{Question}
\newtheorem*{question*}{Question}

\newtheorem*{problem*}{Problem}

\newcommand{\customthmname}{}
\newtheorem{innercustomthm}[thm]{\customthmname}

\newenvironment{customtheorem}[1]
  {\renewcommand{\customthmname}{#1}\begin{innercustomthm}}
  {\end{innercustomthm}}

\newtheorem{thmx}{Theorem}

\theoremstyle{definition}

\newtheorem*{theorem*}{Theorem}

\theoremstyle{remark}
\newtheorem{remark}[thm]{Remark}
\newtheorem*{remark*}{Remark}
\newtheorem{construction}[thm]{Construction}

\newtheoremstyle{custom}{}{}{\itshape}{}{\bfseries}{.}{.5em}{\thmnote{#3}#1}
\theoremstyle{custom}

\usepackage{rotating}

\usepackage{color}
\usepackage{enumerate}

\usepackage{cite}
\usepackage[nodisplayskipstretch]{setspace}
\usepackage{placeins}

\usepackage{listings}
\lstdefinestyle{mystyle}{
    basicstyle=\ttfamily\footnotesize,
    breakatwhitespace=false,         
    breaklines=true,                 
    captionpos=b,                    
    keepspaces=true,                                   
    showspaces=false,                
    showstringspaces=false,
    showtabs=false,                  
    tabsize=2
}
\renewcommand{\l}{\Lambda}

\newenvironment{clist}[1]
{\begin{enumerate}[\normalfont #1]}
{\end{enumerate}}

\newcommand{\wh}{\widehat}
\newcommand{\wt}{\widetilde}

\usepackage{mathrsfs}
\usepackage{adjustbox}
\usepackage{nccmath}
\usepackage{esvect}

\makeatletter
\@namedef{subjclassname@2020}{%
  \textup{2020} Mathematics Subject Classification}
\makeatother

\usetikzlibrary{decorations.markings}

\usepackage{array}

\newcolumntype{P}[1]{>{\centering\arraybackslash}m{#1}}

\begin{document}

\title{Swan modules and homotopy types after a single stabilisation}

\author{Tommy Hofmann} 
\address{Naturwissenschaftlich-Technische Fakult\"{a}t, Universit\"{a}t Siegen,  Walter-Flex-Strasse \newline \indent 3, 57068 Siegen, Germany}
\email{tommy.hofmann@uni-siegen.de}

\author{John Nicholson}
\address{School of Mathematics and Statistics, University of Glasgow, United Kingdom}
\email{john.nicholson@glasgow.ac.uk}

\subjclass[2020]{Primary 55P15; Secondary 20C05, 57Q12}


\begin{abstract}
We study Swan modules, which are a special class of projective modules over integral group rings, and their consequences for the homotopy classification of CW-complexes.
We show that there exists a non-free stably free Swan module, thus resolving Problem A4 in the 1979 Problem List of C.\,T.\,C. Wall.
As an application we show that, in all dimensions $n \equiv 3$ mod $4$, there exist finite $n$-complexes which are homotopy equivalent after stabilising with multiple copies of $S^n$, but not after a single stabilisation. This answers a question of M.\,N. Dyer. 

We also resolve a question of S. Plotnick concerning Swan modules associated to group automorphisms and, as an application, obtain a short and direct proof that there exists a group with $k$-periodic cohomology which does not have free period $k$.
In contrast to the original proof of R.\,J.~Milgram, our proof circumvents the need to compute the Swan finiteness obstruction.
\end{abstract}

\maketitle

\vspace{-5mm}

\section{Introduction}

A \textit{Swan module} over a finite group $G$ is a $\Z G$-module
defined by a left ideal $(N,r) \le \Z G$ where $N = \sum_{g \in G} g$ and $r \in \Z$ is an integer coprime to $|G|$. 
These modules were introduced by Swan \cite{Sw60-II}, who showed that $(N,r)$ is projective and is determined up to isomorphism by the value of $r \bmod |G|$. Consequently, we often write $r \in (\Z/|G|)^\times$.
Swan modules play an important role in topology through the classification of spherical space forms \cite{Sw60,TW71,LT73} and more general CW-complexes \cite{Jo03a, Ni20b, Ni20a}.
They have also found applications in algebraic number theory, particularly through determining the Galois module structure of the ring of integers of a number field \cite{Ta81,Ta84,GRRS99}.

A classical question, which arose in work of Dyer \cite[Note (c) p276]{Dy76} and Plotnick \cite[p98]{Pl82}, is whether every stably free Swan module is free.
A finite group $G$ for which this holds is said to have the \textit{weak cancellation property} \cite[p582]{Jo04}.
The case where $G$ has periodic cohomology, which is most relevant for applications to topology, appeared in Wall's problem list \cite[Problem A4]{Wa79} (see also \cite[Section 2]{Jo04}, \cite[Question 7.2]{Ni20a}).
We will show:

\begin{thmx} \label{thmx:SF-vs-F}
There exists a finite group $G$ and $r \in (\Z/|G|)^\times$ such that $(N, r)$ is a non-free stably free $\Z G$-module. Furthermore, $G$ can be taken to have periodic cohomology.
\end{thmx}

In particular, we can take $G = Q_{56}$ to be the quaternion group of order $56$ and $r=15$.
To find such an example, one can in principle use algorithms of Bley--Johnston for deciding freeness of $\Z G$-lattices \cite{MR2422318,MR2813368}. Whilst this was our initial approach, these computations did not terminate beyond establishing freeness in certain cases. We therefore devised a heuristic procedure (Algorithm \ref{alg}) which either proves freeness or provides heuristic evidence of non-freeness.
We searched across groups of small order and the algorithm singled out $(N,15)$ over $Q_{56}$ as the unique candidate among the cases for which the search terminated (see \cref{appendix:heuristics}). The rigorous proof
that $(N,15)$ is non-free is then achieved as a result of extensive computations of unit groups of rings which arise in two pullback squares (see \cref{eq:square1,eq:square2}).

We say two finite $n$-complexes $X$ and $Y$ are \textit{stably equivalent} if $X \vee kS^n \simeq Y \vee kS^n$ for some $k \ge 0$. Is one stabilisation always enough, i.e.~does $k=1$ always suffice? This question, with various conditions on the CW-complexes, appeared in \cite[Problem (c)]{Dy79a}, \cite[p124]{HMS93} and \cite[Problem B2]{Ni21b}.
The analogous question for simple homotopy equivalence ($\simeq_s$) appeared in \cite[Question 1]{Dy81}.
Using the examples from \cref{thmx:SF-vs-F}, we will show:

\begin{thmx} \label{thmx:CW}
Let $n \ge 1$ with $n \equiv 3 \mod 4$.
Then there exist finite $n$-complexes $X$ and $Y$ such that $X \vee 2S^n \simeq Y \vee 2S^n$ but $X \vee S^n \not \simeq Y \vee S^n$.
Furthermore, we can assume that $X$ and $Y$ have finite fundamental group, $(n-1)$-connected universal covers, and satisfy $X \vee 2S^n \simeq_s Y \vee 2S^n$.
\end{thmx}

The fact that $G$ can be taken to be finite is of interest since it was shown by Browning \cite[Theorem 5.4]{Br78} (see also \cite[Corollary 4.7]{Ni20a}) that, for $n \ge 2$ even, one stabilisation is enough for finite $n$-complexes with finite fundamental group and $(n-1)$-connected universal covers. 
We show that this also holds when $n \equiv 1 \mod 4$ (\cref{thm:all-dimensions}). Hence the dimension restrictions in \cref{thmx:CW} are optimal for finite $n$-complexes satisfying these conditions. 

The examples in \cref{thmx:CW} are obtained by taking $X$ to be the closed $n$-manifold $S^n/Q_{56}$ and defining $Y = X_{15}$ via the following general construction (see \cref{prop:construction-info} for more details):

\begin{construction} \label{construction:Swan}
Let $G$ be a finite group and let $(N,r)$ be a stably free Swan module for some $r \in (\Z/|G|)^\times$. Given a (connected) finite $n$-complex $X$ with fundamental group $G$ and $(n-1)$-connected universal cover, we construct a finite $n$-complex $X_r$ as follows:
\begin{clist}{\hspace{4mm}1.}
\item
Let $m : \Z \to \Z$ be multiplication by a representative of $r^{-1}$, let $\varepsilon : C_0(\wt X) \twoheadrightarrow H_0(\wt X) \cong \Z$ be the standard map and define
\[ m^*(C_*(\wt X)) = (C_n(\wt X) \xrightarrow{\partial_n} \cdots \xrightarrow{\partial_2} C_1(\wt X) \xrightarrow{(\partial_1,0)} C_0(\wt X) \times_{\varepsilon,m} \Z). \]
We have $C_0(\wt X) \times_{\varepsilon,m} \Z \cong (I_G,r^{-1}) \oplus \Z G^k \cong (N,r) \oplus \Z G^k$ for some $k \ge 0$ by \cite[Proposition 3.4 \& Lemma 4.12]{Ni20a}.
This is stably free so, by adding extensions of the form $\Z G \xrightarrow[]{\cong} \Z G$, this is chain homotopy equivalent to an exact sequence $C_*'$ of free $\Z G$-modules.
\item
Since $n \ge 3$, there exists a finite $n$-complex $X_r$ with fundamental group $G$ such that $C_*(\wt X_r)$ is chain homotopy equivalent to $C_*'$ (see \cite[Corollary 8.27]{Jo12a}).
\end{clist}
The complex $X_r$ has fundamental group $G$, $(n-1)$-connected universal cover, $\pi_n(X_r) \cong \pi_n(X)$ as $\Z G$-modules, and $k$-invariant $k^{n+1}(X_r) = r \cdot k^{n+1}(X) \in \Z/|G|$ ($\cong H^{n+1}(G;\pi_n(X))$).
\end{construction}

To prove that $X = S^n/Q_{56}$ and $Y = X_{15}$ are not homotopy equivalent, it remains to show that $1, 15 \in (\Z/56)^\times$ do not coincide under the combined action of $\Aut(Q_{56})$ and $\Aut_{\Z G}(\pi_n(X))$ on the $k$-invariants. Our proof uses that $\pi_n(X) \cong \Z$ and so $\Aut_{\Z G}(\pi_n(X)) = \{\pm \id\}$.
Note that $\Aut_{\Z G}(\pi_n(X))$ is typically hard to compute so, whilst this construction applies to any $X$ and stably free Swan module $(N,r)$, showing that $X_r \not \simeq X$ is difficult in general.

We now give two consequences of \cref{thmx:CW}. The first follows directly from the properties of \cref{construction:Swan} listed above.

\begin{corollary} \label{cor:ThmB-1}
  Let $n > 2$ be even. Then there exist homotopically distinct finite $n$-complexes $X$ and $Y$ with fundamental group $Q_{56}$ and $(n-1)$-connected universal covers such that $\pi_n(X) \cong \pi_n(Y)$ as $\Z Q_{56}$-modules. The same holds for $n = 2$ provided $Q_{56}$ has the {\normalfont D2} property.
\end{corollary}

Such examples have only previously been known for finite abelian fundamental groups \cite{SD79}. The case $n=2$ would require part (2) of \cref{construction:Swan} to apply in that setting, which does so if and only if $Q_{56}$ has the D2 property (see, for example, \cite[Proposition 5.1]{Ni20a}). The D2 property for quaternion groups was studied by the authors in \cite{HN25}. In the notation given there, it would suffice to construct a finite presentation $\mathcal{P}$ for $Q_{56}$ such that $\Psi(\mathcal{P}) = [(N,15)]$.

We next consider applications to manifolds. It is an open problem to determine whether there exist closed smooth $2n$-manifolds $M$ and $N$ such that $M \# 2(S^n \times S^n) \cong N \# 2(S^n \times S^n)$ are diffeomorphic but $M \# (S^n \times S^n) \not \cong N \# (S^n \times S^n)$, with particular interest in the case where $n=2$ and $M$ and $N$ are homeomorphic.
This question admits many variations for manifolds with boundary, and it was shown by Kang that examples exist for homeomorphic compact smooth $4$-manifolds with boundary \cite{Ka22}. 
Let $\natural$ denote boundary connected sum.

\begin{corollary} \label{cor:ThmB-2}
Let $n \ge 1$ with $n \equiv 3 \mod 4$. Then there exist smooth compact $(2n+1)$-manifolds $M$ and $N$ such that $M \, \natural \, 2(S^n \times D^{n+1}) \cong N \, \natural \, 2(S^n \times D^{n+1})$ but $M \, \natural \, (S^n \times D^{n+1}) \not \cong N \, \natural \, (S^n \times D^{n+1})$. In fact, we have $M \, \natural \, (S^n \times D^{n+1}) \not \simeq N \, \natural \, (S^n \times D^{n+1})$.
\end{corollary}

This follows directly from \cref{thmx:CW} by taking $M$ and $N$ to be smooth regular neighbourhoods of embeddings of $X$ and $Y$ into $\R^{2n+1}$ respectively. Then $X \vee kS^n \simeq_s M \, \natural \, k(S^n \times D^{n+1})$ for all $k \ge 0$ (and similarly for $Y$). This implies that $M \, \natural \, (S^n \times D^{n+1}) \not \simeq N \, \natural \, (S^n \times D^{n+1})$. Since $M \, \natural \, 2(S^n \times D^{n+1})$ and $N \, \natural \, 2(S^n \times D^{n+1})$ are both simple homotopy equivalent to $X \vee 2S^n$ and embed in $\R^{2n+1}$ with $2n+1 \ge 7$, they are trivial thickenings and so are diffeomorphic \cite[p76]{Wa66}. 

We have $\partial(M \, \natural \, k(S^n \times D^{n+1})) = \partial M \# k(S^n \times S^n)$ for all $k \ge 0$ (and similarly for $N$) and so $\partial M \# 2(S^n \times S^n) \cong \partial N \# 2(S^n \times S^n)$. This leads to a potential strategy, in dimensions $2n \equiv 6 \mod 8$, for tackling the original problem for closed manifolds. However, determining whether $\partial M \# (S^n \times S^n)$ and $\partial N \# (S^n \times S^n)$ are homotopy equivalent for a given pair $M, N$ appears to be a difficult problem. Further discussion is beyond the scope of this article.

\begin{question}
Let $n \ge 1$ with $n \equiv 3 \mod 4$. Do there exist compact $(2n+1)$-manifolds $M$ and $N$ as in \cref{cor:ThmB-2} such that $\partial M \# (S^n \times S^n) \not \cong \partial N \# (S^n \times S^n)$?
\end{question}

We now give our third main result which, whilst it concerns Swan modules, is not related directly to Theorems \ref{thmx:SF-vs-F} or \ref{thmx:CW}. Recall that a finite group $G$ has \textit{free period $k$} if there exists a $k$-periodic resolution of finitely generated free $\Z G$-modules. Such groups necessarily have $k$-periodic cohomology. The question of whether or not the converse holds originated in work of Swan \cite{Sw60-II} and was a major motivating question during the early development of algebraic K-theory, particularly in light of the fact that finite $3$-manifold groups necessarily have free period $4$. The question featured in Wall's problem list \cite[Problem A3]{Wa79} and was eventually resolved by Milgram \cite{Mi85} who showed that certain groups with $4$-periodic cohomology of the form $Q(2^na,b,c)$ do not have free period $4$ (see also \cite{Da81,DM85}).

If $G$ is a finite group with $k$-periodic cohomology, then $H^k(G;\Z) \cong \Z/|G|$ and the functor $H^k(-;\Z)$ induces a map $\psi_k : \Aut(G) \to (\Z/|G|)^\times$. It was shown independently by Dyer~\cite[Note (b) p276]{Dy76} and Davis~\cite{Da83} that if $G$ has free period $k$, then $(N,\psi_k(\theta))$ is stably free for all $\theta \in \Aut(G)$.
Plotnick \cite[p98]{Pl82} asked whether this holds in general (see also \cite[p488]{Da83}, \cite[Question 7.3]{Ni20a}). 
By work of Bentzen--Madsen \cite[p448]{BM83}, this holds for $Q(8,p,q)$ where $p, q$ are distinct odd primes (see \cite[Corollary 4.4]{BM83}, \cite[p231]{Ma83}).
We will show:

\begin{thmx} \label{thmx:automorphism}
Let $k \ge 1$ with $k \equiv 4 \mod 8$. Then there exists a finite group $G$ with $k$-periodic cohomology and an automorphism $\theta \in \Aut(G)$ such that $(N,\psi_k(\theta))$ is not stably free.
\end{thmx}

In particular, for each $k$, we can take $G = Q(16,5,1)$ and $\theta \in \Aut(G)$ such that $\psi_k(\theta) = 9$.
By combining this with Dyer's result that $(N,\psi_k(\theta))$ is stably free for all $\theta \in \Aut(G)$ provided $G$ has free period $k$, we obtain:

\begin{corollary}
Let $k \ge 1$ with $k \equiv 4 \mod 8$. Then there exists a finite group with $k$-periodic cohomology which does not have free period $k$.
\end{corollary}

This was previously established by Milgram \cite{Mi85}, who used that, if $G$ has $k$-periodic cohomology, there exists an element $\sigma_k(G) \in \wt K_0(\Z G)/\{ [(N,r)] : r \in (\Z/|G|)^\times\}$ known as the \textit{Swan finiteness obstruction} which vanishes if and only if $G$ has free period $k$. 
He used pullback squares to show that $\sigma_4(Q(8,p,q)) \ne 0$ for certain primes $p,q$. 
Using an analogous approach, Davis \cite{Da81} showed that $\sigma_4(Q(16,p,1)) \ne 0$ for certain primes $p$, including the case $p = 5$. The extension to the case $k =4i$ for $i$ odd follows from Wall's theorem that $2 \cdot \sigma_4(G) = 0$ whenever $G$ has $4$-periodic cohomology \cite{Wa79-paper}, combined with the standard fact that $\sigma_{4i}(G) = i \cdot \sigma_4(G)$.

In contrast, we make no use of the Swan finiteness obstruction or Wall's theorem and instead we need only show that $[(N,9)] \ne 0 \in \wt K_0(\Z Q(16,5,1))$.
Whilst it is possible to prove this by hand, we use algorithms of Bley--Boltje \cite{MR2282916} and Bley--Wilson \cite{MR2564571} which explicitly compute $\wt K_0(\Z G)$ for a finite group $G$.
Both algorithms have been implemented in \textsc{Magma}~\cite{MR1484478} and the code we used can be found in \cref{appendix:magma}.

We now give two more applications of \cref{thmx:automorphism}. The first concerns projective resolutions. 

\begin{corollary} \label{cor:theta-resolutions}
Let $k \ge 1$ with $k \equiv 4 \mod 8$.
Then there exists a finite group $G$ and an exact sequence of finitely generated projective $\Z G$-modules
\[ \mathscr{R} = (0 \to \Z  \to P_{k-1} \to P_{k-2} \to \cdots \to P_1 \to P_0 \to \Z \to 0) \] 
whose Euler class $e(\mathscr{R}) = \sum_{i=0}^k (-1)^i [P_i] \in \wt K_0(\Z G)$ does not coincide with $e(\mathscr{R}_\theta)$ for some $\theta \in \Aut(G)$, where $\mathscr{R}_\theta = \{(P_i)_\theta\}_{i=0}^k$ denotes $\mathscr{R}$ with $G$-action reparametrised by $\theta$. 
\end{corollary}

We have $e(\mathscr{R}_\theta) = e(\mathscr{R})_\theta$ and so this equivalently says that $e(\mathscr{R}) \in \wt K_0(\Z G)$ is not invariant under the action of $\Aut(G)$.
The result follows from the fact that $e(\mathscr{R}_\theta) - e(\mathscr{R}) = [(N,\psi_k(\theta)]$ by \cite[Proposition 3.4 \& Section 6.2]{Ni20a}.
This gives the first example where the action of $\Aut(G)$ on projective modules defined in \cite[Theorem B]{Ni20a} does not coincide with the usual action $P \mapsto P_\theta$, which has consequences for the homotopy classification of finite $n$-complexes with $(n-1)$-connected universal cover (see \cite[Section 1.2]{Ni20a} for further discussion).

Our final application resolves a question of Plotnick in the affirmative \cite[p98]{Pl82}.

\begin{corollary}
Let $n \ge 1$ with $n \equiv 2 \mod 8$. Then there exists an $n$-complex $X$ such that not every automorphism of $\pi_1(X)$ is realised as $\pi_1(f)$ for $f$ a (pointed) self homotopy equivalence of $X \vee k S^n$ for some $k \ge 0$, i.e.~$\colim_{k \to \infty} \pi_1(\hAut(X \vee k S^n)) \subsetneq \Aut(\pi_1(X))$. 
\end{corollary}

In particular, we can take $X$ to be the $n$-complex homotopy equivalent to $Y/Q(16,5,1)$ with one point removed, where $Y$ is an $(n+1)$-complex homotopy equivalent to $S^{n+1}$ on which $Q(16,5,1)$ acts freely. The result then follows from the main theorem in \cite{Pl82}.

\subsection*{Organisation of the paper} The article will be structured as follows. 
In \cref{s:swan-modules}, we give the necessary background on Swan modules before studying these modules over quaternion groups $Q_{8p}$ for $p$ an odd prime, leading to a proof of \cref{thmx:SF-vs-F}.
In \cref{s:CW}, we give more details on \cref{construction:Swan} and the homotopy types of CW-complexes, and we prove \cref{thmx:CW} and \cref{thm:all-dimensions}. In \cref{s:Swan-module-automorphisms}, we prove \cref{thmx:automorphism}. In \cref{appendix:heuristics}, we discuss our heuristic procedure for identifying potential candidates for non-free stably free Swan modules. In \cref{appendix:magma}, we record the \textsc{Magma} computations which we use in the proof of \cref{thmx:automorphism}.

\subsection*{Acknowledgements} 
TH gratefully acknowledges support by the Deutsche Forschungsgemeinschaft -- Project-ID 286237555 -- TRR 195; and Project-ID 539387714.
JN was supported by a Rankin-Sneddon Research Fellowship from the University of Glasgow.
We would like to thank Werner Bley for advice on \textsc{Magma} computations, and Jim Davis and Mark Powell for helpful comments on the manuscript.

\section{Swan modules over quaternion groups} \label{s:swan-modules}

The primary aim of this section is to establish the existence of non-free stably free Swan modules, and hence to prove \cref{thmx:SF-vs-F}.
Preliminaries on Swan modules are given in \cref{ss:swan-prelim}.
In \cref{ss:swan-stablyfree} and \cref{ss:swan-free} we study stably free and free Swan modules for quaternion groups respectively. In \cref{ss:swan-proof}, we use this to prove \cref{thmx:SF-vs-F}.

\subsection*{Conventions} 
For a ring $R$, all $R$-modules will be assumed to be finitely generated left $R$-modules. If $M$ is an $R$-module and $f : R \to S$ is a ring homomorphism, the $S$-module obtained via extension of scalars will be denoted by $f_\#(M)$ or $S \otimes_R M$.

\subsection{Preliminaries on Swan modules}
\label{ss:swan-prelim}

Let $G$ be a finite group. Recall from the introduction that a Swan module is a $\Z G$-module of the form $(N,r) = \Z G \cdot N + \Z G \cdot r \le \Z G$ where $N = \sum_{g \in G} g \in \Z G$ and $r \in \Z$ is such that $(r,|G|)=1$. If $r \equiv r' \mod |G|$, then $(N,r) \cong (N,r')$ \cite[Lemma 6.1 (a)]{Sw60-II} and so $(N,r)$ is well-defined for $r \in (\Z/|G|)^\times$. If $r, r' \in (\Z/|G|)^\times$, then $(N,r) \oplus (N,r') \cong (N,rr') \oplus \Z G$ \cite[Lemma 6.1 (c)]{Sw60-II}. This implies that Swan modules are projective since $(N,r) \oplus (N,r^{-1}) \cong \Z G^2$. Hence we obtain a group homomorphism
\[ S : (\Z/|G|)^\times \to \wt K_0(\Z G), \quad r \mapsto [(N,r)] \]
which we will refer to as the \textit{Swan map}. Define the \textit{Swan subgroup} to be $T(G) = \IM(S)$.

To determine when Swan modules are free or stably free, we must determine the sets:
\[ \FF(G) = \{ r \in (\Z/|G|)^\times : (N,r) \cong \Z G\}, \quad 
\SF(G) = \ker(S_G)
\]
which are both subgroups of $(\Z/|G|)^\times$; in the case of $\FF(G)$, this follows from the fact that $(N,r) \otimes (N,r') \cong (N,rr')$.
We have $\FF(G) \le \SF(G)$ and $T(G) \cong (\Z/|G|)^\times/\SF(G)$. By \cite[Lemma 6.3]{Sw60-II}, we have that
\[ \FF(G) = \IM(\varepsilon : (\Z G/N)^\times \to (\Z/|G|)^\times) \]
where $\varepsilon$ is induced by the augmentation map. This can be made explicit as follows. 

\begin{lemma} \label{lemma:ZG/N-unit-to-isomorphism}
Let $G$ be a finite group, let $r \in (\Z/|G|)^\times$ and let $u_r \in (\Z G/N)^\times$ be such that $\varepsilon(u_r) = r$. 
Let $s = r^{-1} \in (\Z/|G|)^\times$, let $t \in \Z$ be such that $rs = 1 + t|G|$ and let $u_s = u_r^{-1} \in (\Z G/N)^\times$.
Let $\wt u_r, \wt u_s \in \Z G$ be the unique representatives of $u_r, u_s \in \Z G/N$ such that $\varepsilon(\wt u_r ) = r$, $\varepsilon(\wt u_s) = s$.
Then there are mutually inverse isomorphisms of $\Z G$-modules:
\[ f : \Z G \to (N,r), \,\, 1 \mapsto \wt u_s r - tN, \qquad g : (N,r) \to \Z G, \,\, N \mapsto N, \, r \mapsto \wt u_r. \]
\end{lemma}

\begin{proof}
We can check directly that $f \circ g = \id$ and $g \circ f = \id$ (see also \cite[Lemma 6.3]{Sw60-II}).
\end{proof}

Let $G$ be a finite group with $k$-periodic cohomology. Then $H^k(G;\Z) \cong \Z/|G|$ (see \cite[XII.11.1]{CE56}) and the functor $H^k(-;\Z)$ induces a map
\[ \psi_k : \Aut(G) \to (\Z/|G|)^\times \]
where we have used the identifications $\Aut(H^k(G;\Z)) \cong \Aut(\Z/|G|) \cong (\Z/|G|)^\times$. We define $\Aut_k(G) := \IM(\psi_k) \le (\Z/|G|)^\times$. 
Since $G$ also has $ik$-periodic cohomology for all $i \ge 1$, we can consider $\psi_{ik}$ for any $i \ge 1$.
The following can be found in \cite[Lemma 6.7]{Ni20a}.

\begin{lemma} \label{lemma:psi_k-properties}
  Let $G$ be a finite group with $k$-periodic cohomology. If $i \ge 1$ and $\theta \in \Aut(G)$, then $\psi_{ik}(\theta) = \psi_k(\theta)^i$. In particular, $\Aut_{ik}(G) = (\Aut_k(G))^i$.
\end{lemma}

\subsection{Stably free Swan modules}
\label{ss:swan-stablyfree}

From now on, $p$ will be an odd prime and we identify $Q_{8p} = \langle x,y \mid x^{2p}=y^2,yxy^{-1}=x^{-1}\rangle$.
We begin by showing the following using the results of Bentzen--Madsen \cite{BM83} which compute $T(Q_{8p})$ for $p$ an odd prime. 
For an integer $m$ coprime to $p$, let $\ord_p(m)$ denote the order of $m$ in $(\Z/p)^\times$ and let $\left(\frac{m}{p}\right) \in \{\pm 1\}$ denote the Legendre symbol for the quadratic residue of $m$ mod $p$.

\begin{prop} \label{prop:SF(Q8p)}
Let $p$ be an odd prime and let $\psi : (\Z/8p)^\times \to (\Z/4)^\times = \{ \pm 1\}$ denote reduction mod $4$. Then:
\[
\SF(Q_{8p}) =  
\begin{cases}
\{r : r \equiv \pm 1 \mod 8 \},  \hspace{17mm} \text{if {\normalfont (}$p \equiv -1 \mod 8${\normalfont)} or {\normalfont (}$p \equiv 1 \mod 8$, $\ord_p(2)$ is odd\,{\normalfont)}}\\
\{r : r \equiv \pm 1 \mod 8, \left(\frac{r}{p}\right) = 1 \},  \,\,\text{if {\normalfont (}$p \equiv -3 \mod 8${\normalfont)} or {\normalfont (}$p \equiv 1 \mod 8$, $\ord_p(2)$ is even{\normalfont)}} \\
\{r : r \equiv \pm 1 \mod 8, \psi(r) \cdot \left(\frac{r}{p}\right) = 1 \}, \,\,\text{if $p \equiv 3 \mod 8$.}
\end{cases}
\]
\end{prop}

\begin{remark}
This contradicts \cite[Theorem 1]{LT73}, which claims erroneously that $(N,r)$ is always stably free over $Q_{8p}$ provided $r \equiv \pm 1 \mod 8$. For example, if $p = 5$ and $r=7 \in (\Z/8p)^\times$, then $r \equiv -1 \mod 8$ but $(N,r)$ is not stably free since $p \equiv -3 \mod 8$ and $\left(\frac{r}{p}\right) = -1$.
\end{remark}

\begin{proof}
In \cite[Theorem 3.5]{BM83}, it is shown that
$T(Q_{8p}) \cong S(1) \oplus S(p)$
where $S(1) \cong \Z/2$ and 
\[
S(p) =
\begin{cases}
\Z/2, & \text{if {\normalfont (}$p \equiv \pm 3 \mod 8${\normalfont)} or {\normalfont (}$p \equiv 1 \mod 8$ and $\ord_p(2)$ is even{\normalfont)}} \\
0, & \text{if {\normalfont (}$p \equiv -1 \mod 8${\normalfont)} or {\normalfont (}$p \equiv 1 \mod 8$ and $\ord_p(2)$ is odd{\normalfont)}.}	
\end{cases}
 \]
Let $S = (S_1,S_2) : (\Z/8p)^\times \to T(Q_{8p}) \cong S(1) \oplus S(p)$ denote the two components of the Swan map. 
In \cite[Addendum 3.6]{BM83}, it is shown that $\ker(S_1) = \{r : r \equiv \pm 1 \mod 8\}$ and that
\[ \ker(S_2) = 
\begin{cases}
	\{ r : \left(\frac{r}{p}\right) = 1\}, & \text{if $p \equiv -3 \mod 8$, or $p \equiv 1 \mod 8$, $\ord_p(2)$ is even} \\
	\{ r : \psi(r) \cdot \left(\frac{r}{p}\right) = 1\}, & \text{if $p \equiv 3 \mod 8$.}
\end{cases}
 \]
The result follows since $\SF(Q_{8p}) = \ker(S) = \ker(S_1) \cap \ker(S_2)$.
\end{proof}

\subsection{Free Swan modules}
\label{ss:swan-free}

We will now compute $\FF(Q_{8p})$.
First note that this coincides with $\IM(\varepsilon : (\Z Q_{8p}/N)^\times \to (\Z/8p)^\times)$ and so it suffices to compute the augmentations of units in $(\Z Q_{8p}/N)^\times$. However, we will take a different approach, simply viewing $(N,r)$ as a projective $\Z Q_{8p}$-module and using pullback square methods to determine when it is free.

In \cite[Section 11]{Sw83}, Swan gave a general strategy for classifying projective $\Z Q_{8p}$-modules using the pullback square: 
\begin{equation}
\begin{tikzcd}
	\Z Q_{8p} \ar[r,"a_2"] \ar[d,"a_1"] & \LL \ar[d,"b_2"] \\
	\Z D_{4p} \ar[r,"b_1"] & \F_2 D_{4p} 
\end{tikzcd}
\label{eq:square1}
\end{equation}
where $\LL = \Z Q_{8p}/(x^{2p}+1)$, where we identify $D_{4p} = \langle x,y \mid x^{2p}=y^2=1,yxy^{-1}=x^{-1}\rangle$, and where each map is the natural map $x \mapsto x$, $y \mapsto y$.
This approach was used to completely classify the projective $\Z Q_{8p}$-modules for $p = 3, 5$ \cite[Theorem 11.14]{Sw83}.

By Milnor patching \cite{Mi71} applied to the diagram in~\cref{eq:square1}, a projective $\Z Q_{8p}$-module $P$ with $(a_1)_\#(P) \cong \Z D_{4p}$ and $(a_2)_\#(P) \cong \LL$ has an associated unit $u_P \in \F_2 D_{4p}^\times$ such that $P$ is determined up to isomorphism by the image of $u_P$ in the double coset quotient $\Z D_{4p}^\times \backslash \F_2 D_{4p}^\times /{\LL^\times}$.
The aim of this section will be to establish the following.

\begin{prop}\label{prop:swansetup}
Let $p$ be an odd prime and let $r \in \Z$, $\gcd(r, 8p) = 1$.
Let $\zeta_p$ be a primitive $p$th root of unity, let $\lambda_p = \zeta_p+\zeta_p^{-1}$, let $R_p = \Z[\lambda_p]$ and let $\mathfrak p$ be the unique prime ideal of $R_p$ over $p$.
\begin{clist}{(i)}
\item
    We have $(a_1)_\#((N, r)) \cong \Z D_{4p}$ and $(a_2)_\#((N, r)) \cong \Lambda$.
\item
Let $s,t \in \Z$, $s \geq 1$ be such that $rs = 1+t|G|$ and $m = (s-1)/2$. The isomorphism class of $(N, r)$ is determined by the double coset $[u_{(N, r)}] \in \Z D_{4p}^\times \backslash \F_2 D_{4p}^\times /{\LL^\times}$, where
\[ u_{(N,r)} = 1+(1+x^r + \dotsb +  (x^r)^{m-1})x^{r-1}(x+y) + tN \in \F_2 D_{4p}^\times. \]
\item
There exist maps $\rho_1 \colon \F_2 D_{4p}^\times \to \Z/2$, $\rho_2 \colon \F_2 D_{4p}^\times \to (R_p/2R_p)[C_2]^\times/(1 + \mathfrak p R_p C_2)$ inducing a bijection
\[
\overline \rho \colon \Z D_{4p}^\times \backslash \F_2 D_{4p}^\times /{\LL^\times}
\longrightarrow \left(\Z/2 \oplus \frac{(R_p/2R_p)[C_2]^\times}{1+\mathfrak{p}R_pC_2}\right) \Big/ (\rho_1, \rho_2)(\Lambda^\times), \quad [u] \mapsto [(\rho_1(u),\rho_2(u))]
\] 
with $\overline \rho([1]) = 0$.
In particular, $(N, r)$ is free if and only if $\overline \rho([u_{(N, r)}]) = 0$.
\item
We have
\[ \rho_1(u_{(N,r)}) = 
\begin{cases}
0, & r \equiv \pm 1 \mod 8 \\
1, & r \equiv \pm 3 \mod 8. 
\end{cases}
\]
\end{clist}
\end{prop}

The rest of this section is devoted to the proof of Proposition~\ref{prop:swansetup}.
We begin by determining free generators over $\Z D_{4p}$ for which will use the following result.

\begin{lemma} \label{lemma:ZG/N-unit-dihedral}
Let $n \ge 2$ and let $D_{2n} = \langle x,y \mid x^{n}=y^2=1,yxy^{-1}=x^{-1}\rangle$ denote the dihedral group of order $2n$. For each $r \ge 1$ such that $(r,2n)=1$ write $k = (r-1)/2$ and define
\[ u_r = 1+x+ \cdots + x^k + y + xy + \cdots + x^{k-1}y \in \Z D_{2n} / N_{D_{2n}}. \]
Then
$u_r$ depends only on $r \mod 2n$ (i.e.~on $k \mod n$) and
$u_r \in (\Z D_{2n} / N_{D_{2n}})^\times$ is a unit with inverse
\[ u_r^{-1} = 1+(1+x^{2k+1}+(x^{2k+1})^2 + \dotsb + (x^{2k+1})^{m-1})x^{2k}(x+y), \]
where $m = (s-1)/2$ and $s \ge 1$ is such that $rs \equiv 1 \mod 2n$.  We have $\varepsilon(u_r) = r \in (\Z/2n)^\times$.
\end{lemma}

\begin{remark}
That $u_r \in (\Z D_{2n} / N_{D_{2n}})^\times$ is a unit is an exercise in \cite[Exercise 53.3]{CR87}. However, the identities listed in the hint below the exercise are false and so would lead to an incorrect formula for $u_r^{-1}$.
\end{remark}

\begin{proof}
To show $u_r$ depends only on $r \mod 2n$, we can verify directly that $u_{r+2n} = u_r + N \equiv u_r \in \Z D_{2n} / N_{D_{2n}}$ for all $r$.
We now claim that, for $s \ge 1$ such that $(s,2n)=1$ and $m = (s-1)/2$, we have
\[ u_{rs} = (1+(1+x^{2k+1}+(x^{2k+1})^2 + \dotsb + (x^{2k+1})^{m-1})x^{2k}(x+y))u_r. \]
This implies the formula for $u_r^{-1}$ since $rs \equiv 1 \mod 2n$ implies that $u_{rs} = u_1 = 1$.

For each $i \ge 1$, let $\Sigma_i = 1+ x+ \dotsb + x^i \in \Z D_{4p}$ and $\Sigma_i^{-1} = 1+ x^{-1} + \dotsb + x^{-i}$. Note that $u_r = \Sigma_k + \Sigma_{k-1}y$. Using the identities $y \Sigma_i = \Sigma_i^{-1} y$, $x^{i}\Sigma_i^{-1} = \Sigma_i$ and $\Sigma_i+x^{i+1} \Sigma_j = \Sigma_{i+j+1}$, we have
\[ x^{k}(x+y)u_r = x^{k+1}(\Sigma_k + \Sigma_{k-1}y) + x^{k}(\Sigma_k^{-1}y + \Sigma_{k-1}^{-1}) = \Sigma_{2k}(x+y). \]	
By summing this identity, it follows that
\[ (1+(1+x^{2k+1}+(x^{2k+1})^2 + \dotsb +  (x^{2k+1})^{m-1})x^{2k}(x+y))u_r = \Sigma_{k+m(2k+1)} + \Sigma_{k+m(2k+1)-1}y = u_{rs} \]  	
since $2(k+m(2k+1))+1 = (2k+1)(2m+1) = rs$.	
\end{proof}

We will now prove \cref{prop:swansetup}~(i)~and~(ii).

\begin{lemma}
Let $s,t \in \Z$, $s \geq 1$ be such that $rs = 1+t|Q_{8p}|$ and $m = (s-1)/2$. The isomorphism class of $(N, r)$ is determined by the double coset $[u_{(N, r)}] \in \Z D_{4p}^\times \backslash \F_2 D_{4p}^\times /{\LL^\times}$, where
\[ u_{(N,r)} = 1+(1+x^r + \dotsb +  (x^r)^{m-1})x^{r-1}(x+y) + tN \in \F_2 D_{4p}^\times. \]
\end{lemma}

\begin{proof}
We will use the version of Milnor patching for ideals given in \cite[Lemma 2.3]{HN25}. 
We have $a_2((N,r)) = \LL \cdot N + \LL \cdot r = \LL \cdot r \le \LL$ since $N = (1+x^{2p})(1+ x + \cdots + x^{2p})(1+y) = 0 \in \LL$ and $a_1((N,r)) = (2N_{D_{4p}},r) \le \Z D_{4p}$ where $N_{D_{4p}} = \sum_{g \in D_{4p}} g$. Since $r$ is odd, we can write $r=2k+1$ so that 
\[N_{D_{4p}} = r N_{D_{4p}} - k(2N_{D_{4p}}) \in (2N_{D_{4p}},r)\]
and so $a_1((N,r)) = (N_{D_{4p}},r)$ is a Swan module. 

Let $s,t \in \Z$, $s \geq 1$ be such that $rs = 1+t|Q_{8p}|$, let $u_r \in (\Z D_{4p}/N_{D_{4p}})^\times$ be as defined in \cref{lemma:ZG/N-unit-dihedral}, let $u_s = u_r^{-1} \in (\Z D_{4p}/N_{D_{4p}})^\times$ and let $\wt u_s \in \Z D_{4p}$ be the unique lift of $u_s$ such that $\varepsilon(\wt u_s) = s \in \Z$.
By \cref{lemma:ZG/N-unit-to-isomorphism}, we have that 
$(N_{D_{4p}},r) = \Z D_{4p} \cdot (\wt u_s r - tN)$.
Since $r$ is odd, we have that  $b_2(r) = 1 \in \F_2 D_{4p}^\times$. Thus the conditions of \cite[Lemma 2.3]{HN25} are satisfied and $(N,r)$ is determined by $[u_{(N, r)}] \in \Z D_{4p}^\times \backslash \F_2 D_{4p}^\times /{\LL^\times}$ where $u_{(N,r)} = b_1(\wt u_s r - tN)b_2(r)^{-1} = \wt u_s + tN$. The result now follows from the formula for $\wt u_s$ given in \cref{lemma:ZG/N-unit-dihedral}.
\end{proof}

We now determine the maps $\rho_1$ and $\rho_2$ used in the description of the double coset quotient $\Z D_{4p}^\times \backslash \F_2 D_{4p}^\times /{\LL^\times}$ as in \cref{prop:swansetup}~(iii).
Since $D_{4p}$ satisfies the Eichler condition, \cite[Theorem A18]{Sw83} implies that $b_1(\Z D_{4p}^\times)$ is normal in $\F_2 D_{4p}^\times$. Thus we obtain a bijection 
\[ \Z D_{4p}^\times \backslash \F_2 D_{4p}^\times /{\LL^\times} \xrightarrow[]{\cong} \coker \left(\LL^\times \to {\F_2D_{4p}^\times}/{\Z D_{4p}^\times}\right). \]

We now give an explicit description of $\F_2 D_{4p}^\times/\Z D_{4p}^\times$, following~\cite{Sw83}.
Let $V = C_2^2 = \langle x, y \mid x^2=y^2=1,[x,y]=1 \rangle$ denote the Klein four group. Recall that there is a surjective group homomorphism $f : D_{4p} \twoheadrightarrow V$, $x \mapsto x$, $y \mapsto y$ and that
we define $\lambda_p = \zeta_p+\zeta_p^{-1}$ and $R_p = \Z[\lambda_p]$. 
Consider the subring $Z = \Z[x+x^{-1}] \le \Z D_{4p}/(\Phi_{4p}(x^2))$ and consider the ideal $\mathfrak{p} = (\lambda_p^2-4) \le R_p$, which is the unique prime ideal of $R_p$ over $p$. 

There are canonical ring homomorphisms
\[ \F_2 D_{4p} \twoheadrightarrow M_2(Z/2Z), \quad x \mapsto \left(\begin{smallmatrix} z & 1\\ 1 & 0 \end{smallmatrix}\right), \, y \mapsto \left(\begin{smallmatrix} 0 & 1\\ 1 & 0 \end{smallmatrix}\right) \] 
 and
\[ Z \xrightarrow[]{\cong} R_p C_2, \quad \alpha(z) \mapsto \left(\frac{\alpha(\lambda_p) + \alpha(-\lambda_p)}{2}\right) + \left(\frac{\alpha(\lambda_p) - \alpha(-\lambda_p)}{2}\right)t, \]
where we take $C_2 = \langle t \mid t^2=1 \rangle$.
Finally, let $\rho_1$ and $\rho_2$ be the compositions
\begin{equation}\label{eq:rhos}
\begin{aligned}
\rho_1 &: \F_2 D_{4p}^\times \xrightarrow[]{f_*} \F_2 V^\times \twoheadrightarrow {(\F_2 V)^\times}/V = \Z/2 \cdot (1+x+y) \cong \Z /2 ,
\\
\rho_2 &: \F_2 D_{4p}^\times \to \GL_2(Z/2Z) \xrightarrow[]{\det} (Z/2Z)^\times \xrightarrow[]{\cong} (R_p/2R_p)[C_2]^\times  \twoheadrightarrow \frac{(R_p/2R_p)[C_2]^\times}{1+\mathfrak{p}R_pC_2}.
\end{aligned} 
\end{equation}

\begin{lemma}\label{prop:rho1rho2}
The maps $\rho_1$ and $\rho_2$ induce an isomorphism of abelian groups
\[
(\overline \rho_1, \overline \rho_2) \colon \F_2 D_{4p}^\times/\Z D_{4p}^\times
\xrightarrow[]{\cong} \Z/2 \oplus \frac{(R_p/2R_p)[C_2]^\times}{1+\mathfrak{p}R_pC_2}.
\]
\end{lemma}

\begin{proof}
The statement follows from the proof of \cite[Lemma 11.1]{Sw83}, where the maps are constructed using~\cite[Lemma 10.5]{Sw83}.	
\end{proof}

The following establishes item~(iv) of~\cref{prop:swansetup}, thus completing the proof.

\begin{lemma}\label{lem:rho1u}
Let $r \in \Z$ with $(r,8p)=1$. Then
\[ \rho_1(u_{(N,r)}) = 
\begin{cases}
0, & r \equiv \pm 1 \mod 8 \\
1, & r \equiv \pm 3 \mod 8. 
\end{cases}
 \]
\end{lemma}

\begin{proof}
Let $s, t \in \Z$, $s \geq 1$ with $rs = 1 + t \cdot \lvert G \rvert$ and define $m = (s - 1)/2$.
Since $N = (1 + x + y + xy)$ and $x^r = x$ in $\F_2[V]$, we have $u_{(N, r)} = 1 + (\sum_{i=0}^{m-1} x^i) \cdot (x + y) + t(1 + x + y + xy) \in \F_2[V]$.
This depends only on $m \bmod 4$ and $t \bmod 2$, and in fact
\[ \rho_1(u_{(N, r)}) = 0 \ \Leftrightarrow \ (m \mod 4, t \mod 1) \in \{ (0, 0), (1, 1), (2, 1), (3, 0)\}.\]
Since $\lvert G \rvert = 4p$ we have $rs \equiv 1 + 4t \mod 8$. If $r \equiv 1 \bmod 8$, then $m = (s - 1)/2 \equiv 2t \bmod 8$. Hence $m \equiv 0, 2 \bmod 4$ and $t$ is even and odd respectively. Thus we have $\rho_2(u_{(N, r)}) = 0$. The other cases follow analogously.
\end{proof}

\subsection{Proof of \cref{thmx:SF-vs-F}}
\label{ss:swan-proof}

We will now establish the following, which implies \cref{thmx:SF-vs-F}.

\begin{thm}\label{cor:N15}
The Swan module $(N, 15)$ of $\Z[Q_{56}]$ is a non-free stably free $\Z[Q_{56}]$-module.
\end{thm}

The aim is to apply the criterion from~\cref{prop:swansetup}~(iii), which requires evaluating $u_{(N, r)}$ under the maps $\rho_1$ and $\rho_2$ from \cref{eq:rhos} and determining $\rho_i(\Lambda^\times)$ for $i = 1,2$. This will be done through a series of lemmas.
Some apply to $Q_{8p}$ for $p$ arbitrary (\cref{lemma:helpergeneral,lemma:helpergeneral2}) but, whilst we would like to obtain a general statement for $\FF(Q_{8p})$, we restrict certain calculations to the case $p = 7$ (\cref{lemma:p7,lemma:rho2u,lemma:p7_2,lemma:rho2units}) due to number theoretic obstacles.

We begin by determining
${(R_p/2R_p)[C_2]^\times}/(1+\mathfrak{p}R_pC_2)$.
Let $F = R_p/2R_p$ and let $R_pC_2 \le \Gamma \le \Q(\lambda_p)[C_2]$ be the unique maximal $R_p$-order.

\begin{lemma}\label{lemma:helpergeneral}
Assume that $\langle 2 \rangle = (\Z/p\Z)^\times/\{\pm 1\}$.
The following maps are isomorphisms:
\begin{clist}{(i)}
\item
$F^\times \times F \to F[C_2]^\times, \, (a, b) \mapsto a + ab(1 + t)$
 with inverse $a + bt \mapsto (a + b, b/(a + b))$.
\item
$(\Gamma/2\Gamma)^\times \to F^\times \times F^\times, \, \overline{a + bt} \mapsto (\overline{a + b}, \overline{a - b})$.
    \item
$(\Gamma/2\Gamma)^\times/(R_p[C_2]/2\Gamma)^\times \to F^\times, \, \overline{a + bt} \to \overline{a + bt}\cdot \overline{a - bt}^{-1}$.
\end{clist}
\end{lemma}

\begin{proof}
The assumption implies that $2$ is inert in $R_p$ and that $F$ is a finite field of order $2^{(p-1)/2}$.

(i) This follows from $R_2[C_2] \cong R_2[t]/(1 + t^2)$ and the fact that there is a split exact sequence
\[ 1 \to (1 + t)/(1 + t)^2 \to (F[t]/(1 + t)^2)^\times \to (F[t]/(1 + t))^\times \to 1. \]
  
(ii) This follows using the isomorphism $\Gamma \to R_p \times R_p$, $a + bt \mapsto (a + b, a - b)$.
  
(iii) Note that $2\Gamma$ is generated by $1 + t, 1 - t$ and hence $F \to R_p[C_2]/2 \Gamma$, $\overline a \mapsto \overline a$ is an isomorphism. Thus the image of $(R_p[C_2]/2\Gamma)^\times$ under the isomorphism in (ii) is the diagonal of $F^\times \times F^\times$.
\end{proof}

\begin{lemma}\label{lemma:p7}
For $p = 7$ the following holds:
\begin{clist}{(i)}
\item
$\ker(R_p[C_2]^\times \to (R_p/\mathfrak p R_p)[C_2]^\times) = (g_1 g_3^3 g_5, g_2 g_3^{18}g_4 g_6, g_3^{21}, g_3^9 g_6^3)$, where $g_1 = t$, $g_2 = -t$,
\[ g_{3} = \frac{\lambda_p + 1}2 + \frac{\lambda_p - 1}2 t, \,\, g_{4} = \frac{\lambda_p + 1}2 - \frac{\lambda_p - 1}2 t, \,\, g_{5} = \frac{\lambda_p^2}2 + \frac{\lambda_p^2 - 2}2 t \,\,\,\, \text{\normalfont and} \,\,\,\, g_{6} = \frac{\lambda_p^2}2 - \frac{\lambda_p^2 - 2}2 t.\]
\item
The map
\[ {(R_p/2R_p)[C_2]^\times}/({1+\mathfrak{p}R_pC_2}) \to \Z/2\Z, \, \overline{a + b t} \mapsto c \]
where $b/(a + b) = c_0 + c_1 \overline \lambda + c_2 \overline \lambda^2$ and $c = c_0 + c_1 + c_2$, is an isomorphism.
\end{clist}
\end{lemma}

\begin{proof}
(i) Let $\Gamma$ be the maximal order of $R_p[C_2]$. As $h_p^+ = 1$ (see for example~\cite{Linden1992}), we have $R_p^\times = \langle -1, \lambda_p, \lambda_p^2 - 1 \rangle$ by \cite[Lemma 8.1, Theorem 8.2]{Washington1997}.
Since $R_p \times R_p \to \Gamma$, $(a, b) \mapsto (a + b)/2 + (a - b)/2 \cdot t$ is an isomorphism,
the group $\Gamma^\times$ is generated by $g_i$, $1 \leq i \leq 6$.
By~\cite[Satz 10]{HK1994} (see also~\cite[Proposition 6.8]{JP2020}) we have
    \[ R_p[C_2]^\times = \ker(\Gamma^\times \to (\Gamma/2\Gamma)^\times/(R_p[C_2]/2\Gamma)^\times). \]
    Lemma~\ref{lemma:helpergeneral}~(iii) provides an explicit isomorphism
    $(\Gamma/2\Gamma)^\times/(R_p[C_2]/2\Gamma)^\times \cong F^\times$
    and the images of the generators of $\Gamma^\times$ under this map are
    \[ \overline 1, \overline 1, \overline \lambda_p, \overline \lambda_p^2 + 1 \, (\,= \overline{\lambda_p}^3), \overline{\lambda_p}^6, \overline{\lambda_p}^4. \]
    Since $F^\times = \langle \overline{\lambda_p} \rangle$, a routine kernel computation implies that
    $R_p[C_2]^\times = \langle g_1, g_2, g_3^7, g_3^4 g_5, g_3 g_4, g_3^3 g_6 \rangle$.
    
As $R_p/\mathfrak p R_p \cong \F_p$ has characteristic coprime to $2$, we know that
\[ (R_p/\mathfrak p R_p)[C_2]^\times \to (R_p/\mathfrak p R_p)^\times \times (R_p/\mathfrak p R_p)^\times, \, a + bt \mapsto (a + b, a - b) \]
is an isomorphism. Now $\lambda_p \equiv 2$ modulo $\mathfrak p$ and $3$ is a primitive root modulo $7$. The images of the generators of $R_p[C_2]^\times$ computed before are equal to
\[ (3^3, 1), (1, 3^3), (3^2, 1), (3, 1), (1, 3^2), (1, 3) \]
and hence the claim follows again by a kernel computation.

(ii) We have $F = R_p/2R_p \cong \F_{8}$. By \cref{lemma:helpergeneral} we know that $F[C_2]^\times \to F^\times \times F, \ a + bt \to (a + b, b/(a + b))$
is an isomorphism. Furthermore, $F^\times = \langle \overline{\lambda_p} \rangle$ and $F = \F_2[\overline \lambda_p]$.
The images of the generators of $\ker(R_p[C_2]^\times \to (R_p/\mathfrak p R_p)[C_2]^\times)$ determined in (i) under this map are
\[ (1, 1 + \overline \lambda_p^2), (\overline \lambda_p^4, 1 + \overline \lambda_p), (1, \overline \lambda_p + \overline \lambda_p^2), (\overline \lambda_p^5, 0), \]
which generate the subgroup of $F^\times \times F$ containing all elements of the form $(d, c_0 + c_1 \overline\lambda_p + c_2 \overline \lambda_p^2)$ with $d \in F^\times$ and $c_0 + c_1 + c_2 = 0$.
\end{proof}

Using this description of the codomain of $\rho_2$, we can now evaluate $u_{(N, r)}$ under $\rho_2$.
In what follows, (i) will be used in the proof of \cref{cor:N15}. However, (ii) will not be used until the proof of \cref{thmx:CW} (see \cref{cor:N9}) and is established here only for convenience.

\begin{lemma}\label{lemma:rho2u}
The following holds over $Q_{56}$:
\begin{clist}{(i)}
    \item
        $\rho_2(u_{(N, 15)}) \neq 0$.
    \item
        $\rho_2(u_{(N, 9)}) = 0$.
\end{clist}
\end{lemma}

\begin{proof}
    Let $p = 7$ and $z = x + x^{-1} \in \Z[D_{4p}]/\Phi_p(x^2)$. Then $\Z[z]$ has rank $6$ with $z$ satisfying $z^6 - 5z^4 + 6z^2 - 1 = 0$ and $z^6 + z^5 + 1 = 0$ in $Z/2Z = \F_2[z]$.
    
    (i)
    Note that $15 \cdot 15 = 1 + 8 \cdot 28$ and hence we choose $r = 15$, $t = 8$ and $m = 7$.
    In particular, $x^r = x$ and $t \cdot N = 0$ in $\F_2[D_{4p}]$.
    Thus
    \[ u_{(N, 15)} = 1 + (1 + x + x^2 + \dotsb + x^6) \cdot (x + y) \in \F_2D_{4p}^\times.\] 
    A quick calculation shows that
    \[ u_{(N, 15)} \mapsto \begin{pmatrix} z + z^5 & 0 \\ 1 + z + z^5 & 1 \end{pmatrix} \in \GL_2(\F_2[z]), \]
    which has determinant $z + z^5$, mapping to $(\overline \lambda_p + \overline \lambda_p^5) t = t \in (R_p/2R_p)[C_2]^\times$.
    The claim follows from Lemma~\ref{lemma:p7}~(iii).
    
    (ii)  We can choose $s = 25$ and $t = 8$, so that $u_{(N, 9)} = N + x^5 + x^4y + x^{13}y$, which gets mapped to
    \[ \begin{pmatrix} 1 + z + z^3 + z^5 & 0 \\ z + z^4 & 1 \end{pmatrix} \in \GL_2(\F_2[z]). \]
    The determinant is $1 + z + z^3 + z^5$, and hence $\rho_2(u_{(N, 9)}) = 1 + \overline \lambda_p^2 t$. The claim follows again from Lemma~\ref{lemma:p7} since $\overline \lambda_p^2/(1 + \overline \lambda_p^2) = \overline \lambda_p + \overline \lambda_p^2$.
\end{proof}

We now determine generators of $\Lambda^\times$ using a pullback square for $\Lambda$ and show that these vanish under $\rho_2$.
The general strategy works for any odd prime $p$,
whereas for the computation of specific unit groups of rings of integers in number fields we restrict to $p = 7$.

Recall that $\LL = \Z Q_{8p}/(x^{2p}+1)$, where $p$ is an odd prime. In order to compute $\LL^\times$, we will use the following pullback square (see, for example, \cite[p115]{Sw83}):
\begin{equation}
\begin{tikzcd}
	\LL \ar[r,"c_2"] \ar[d,"c_1"] & \Z[\zeta_{4p},j] \ar[d,"d_2"] \\
	\Z[i,j] \ar[r,"d_1"] & \F_p[i,j] 
\end{tikzcd} \quad 
\begin{tikzcd}
	x,y \ar[r,mapsto] \ar[d,mapsto] & \zeta_{4p},j \ar[d,mapsto] \\
	i,j \ar[r,mapsto] & i,j 
\end{tikzcd}
\label{eq:square2}
\end{equation}
where $\Z[\zeta_{4p},j]$, $\Z[i,j] \le \H_{\R}$ and $\F_p[i,j] := \F_p \otimes_{\Z} \Z[i,j]$. This is induced by the factorisation $x^{2p}+1 = \Phi_4  \Phi_{4p}$ where $\Phi_4 = x^2+1$ and $\Phi_{4p} = x^{2(p-1)}-x^{2(p-2)}+ \cdots -x^2+1$, and the fact that $(\Phi_4,\Phi_{4p}) = (\Phi_4,p)$. In particular, there are ring isomorphisms $\Z[i,j] \cong \l/(\Phi_4)$, $\Z[\zeta_{4p},j] \cong \l/(\Phi_{4p})$ and $\F_p[i,j] \cong \l/(\Phi_4,\Phi_{4p})$. 

Let $N : \Z[i,j]^\times \to \Z^\times$ and $N_{\F_p} : \F_p[i,j]^\times \to \F_p^\times$ denote the quaternionic norms induced by $a+bi+cj+dij \mapsto a^2+b^2+c^2+d^2$. 
For each norm, the preimage $\alpha = a+bi+cj+dij$ of a unit is a unit with inverse given by its conjugate $\bar{\alpha} = a-bi-cj-dij$. Thus we have:
\[ \Z[i,j]^\times = \{\pm 1, \pm i, \pm j, \pm ij\}, \quad \F_p[i,j]^\times =  \{a+bi+cj+dij : a^2+b^2+c^2+d^2 \in \F_p^\times \}. \]
By the pullback square above, there is an isomorphism 
\[\LL^\times \cong \{ (a,b) \in \Z[i,j]^\times \oplus \Z[\zeta_{4p},j]^\times : d_1(a) = d_2(b) \in \F_p[i,j]^\times\} \]
induced by $(c_1,c_2) : \LL \to \Z[i,j] \times \Z[\zeta_{4p},j]$.

\begin{lemma}\label{lemma:helpergeneral2}
The following holds:
\begin{clist}{(i)}
    \item
We have $\l^\times = \{\pm 1, \pm x, \pm y, \pm xy\} \cdot \{ \beta \in \l^\times : c_1(\beta)=1\}$.
\item
    If $S \le \Lambda^\times$ satisfies $c_1(S) = \{1 \}$ and $\langle c_2(S) \rangle = \ker(\Z[\zeta_{4p}]^\times \to \F_p[i])$, then $S$ is a generating set of $\{ \beta \in \l^\times : c_1(\beta)=1\}$.
\end{clist}
\end{lemma}

\begin{proof}
(i) Since $\{\pm 1, \pm x, \pm y, \pm xy\} \le \l^\times$, we have that $\{\pm 1, \pm x, \pm y, \pm xy\} \cdot \{ \beta \in \l^\times : c_1(\beta)=1\} \le \l^\times$ and so it suffices to prove the converse inclusion.

Let $s : \Z[i,j] \to \l$ denote the function $a +bi +cj +dij \mapsto a+bx+cy+dxy$, which is a set theoretic splitting of $c_1$.
The map $c_1 : \l^\times \to \Z[i,j]^\times$ is surjective since $s(\Z[i,j]^\times) = \{\pm 1, \pm x, \pm y, \pm xy\} \le \l^\times$, i.e.~$\{\pm 1, \pm x, \pm y, \pm xy\} \le \l^\times$ maps onto $\Z[i,j]^\times = \{\pm 1, \pm i, \pm j, \pm ij\}$.
If $\alpha \in \l^\times$, then $s(c_1(\alpha^{-1}))\alpha \in \{ \beta \in \l^\times : c_1(\beta)=1\}$ since $c_1(s(c_1(\alpha^{-1}))\alpha) = c_1(\alpha^{-1}) c_1(\alpha) = 1$. Since 
\[ s(c_1(\alpha^{-1})) \le s(\Z[i,j]^\times) = \{\pm 1, \pm x, \pm y, \pm xy\},\] 
we have that $\alpha \in \{\pm 1, \pm x, \pm y, \pm xy\} \cdot \{ \beta \in \l^\times : c_1(\beta)=1\}$ as required.

(ii)
We first show that $\ker(d_2 \colon \Z[\zeta_{4p, j}]^\times \to \F_p[i, j]^\times) = \ker(\Z[\zeta_{4p}]^\times \to \F_p[i])$. 
By \cite[Lemma 7.5 (b)]{MOV83}, we have $\Z[\zeta_{4p},j]^\times = \langle \Z[\zeta_{4p}]^\times, j\rangle$.
If $\alpha \cdot j^{k} \in \Z[\zeta_{4p}, j]^\times$, $\alpha \in \Z[\zeta_{4p}]^\times$, $0 \leq k \leq 3$, is contained in $\ker(d_2)$, then $d_2(j)^k = d_2(\alpha)^{-1} \in \F_p[i]^\times \cap \langle j \rangle = \{ \pm 1 \}$, that is, $k \in \{0,2\}$. It follows that $\alpha \cdot j^k \in \Z[\zeta_{4p}]^\times$.
If $\gamma \in \{ \beta \in \Lambda^\times : c_1(\beta) = 1\}$, then $d_2(c_2(\gamma)) = 1$, that is, $c_2(\gamma) \in \ker(\Z[\zeta_{4p}, j]^\times \to \F_p[i, j]) = \langle c_2(S) \rangle$. Thus the claim follows.
\end{proof}

\begin{lemma}\label{lemma:p7_2}
For $p = 7$ the kernel of the map
    \[ \Z[\zeta_{4p}]^\times \to \F_p[i]^\times, \, \zeta_{4p} \mapsto i\]
    is generated by $\alpha_6^4, \alpha_1^{22}\alpha_2, \alpha_1^{14}\alpha_3^4,\alpha_1^{22}\alpha_3\alpha_4,\alpha_1^{10}, \alpha_3\alpha_5,\alpha_1^{12}\alpha_3^6\alpha_6^{27}$,
    where
    \[ \alpha_1 = 1 + \zeta_p + \zeta_p^2, \alpha_2 = 1 + \zeta_p, \alpha_3 = 1 - \zeta_{4p}^{25}, \alpha_4 = 1 - \zeta_{4p}^{19}, \alpha_5 = 1 - \zeta_{4p}, \alpha_6 = \zeta_{4p}.\]
\end{lemma}

\begin{proof}
By~\cite[Corollary 4.13]{Washington1997} it follows that $\Z[\zeta_{4p}]^\times = \langle \Z[\lambda_{4p}]^\times, 1 - \zeta_{4p}, \zeta_{4p}\rangle$.
The elements $\alpha_i$, $1 \leq i \leq 4$ are independent generators of the cyclotomic units of $\Z[\lambda_{4p}]$ constructed in~\cite{RK1989}. As the cyclotomic units are equal to the unit group in this case by \cite[Theorem]{Sinnott1978} (using that $h_{4p}^+ = 1$ by~\cite{Linden1992}), we conclude that $\Z[\zeta_{4p}]^\times = \langle \alpha_i \mid 1 \leq i \leq 6 \rangle$.
    Since $1$ is not a square modulo $p$, $\F_p[i]$ is a finite field.
    The images of the generators are $3 = (1 + i)^4$, $2 = (1 +i)^8$, $1 - i = (1 + i)^7, 1 +i, 1 -i$ and $i = (1 + i)^8$ respectively. As $\langle 1 + i \rangle \leq \F_p[i]^\times$ is a subgroup of order $24$, the claim follows from a kernel computation.
\end{proof}

\begin{lemma}\label{lemma:rho2units}
  For $p = 7$ we have $\rho_2(\Lambda^\times) = \{0\}$.
\end{lemma}

\begin{proof}
  First note that $\rho_2(-1) = \rho_2(x) = \rho_2(y) = 0$ and hence from Lemma~\ref{lemma:helpergeneral2}~(ii) it follows that $\rho_2(\Lambda^\times) = \rho_2(\{ \beta \in \LL^\times : c_1(\beta) = 1\})$.
  Denote by $\alpha_1,\dotsc,\alpha_6$ and $\beta_1,\dotsc,\beta_6$ the generators of $\Z[\zeta_{4p}]^\times$ and $\ker(\Z[\zeta_{4p}]^\times \to \F_p[i]^\times)$ respectively from Lemma~\ref{lemma:p7_2}.
  By construction, there exist preimages $\tilde \beta_i$ of the $\beta_i$ under $c_2$ with $\tilde \beta_i \in \Lambda^\times$ and $c_1(\tilde{\beta_i}) = 1$. In particular, by Lemma~\ref{lemma:helpergeneral2}~(ii), such preimages $\tilde \beta_i$ generate $\Lambda^\times$.
  Since $\ker(c_2) = (\Phi_{4p}(x))$ and $\rho_2(\Phi_{4p}(x)) = 0$, it suffices to show that every $\beta_i$ has a preimage under $c_2$ that is mapped to zero under $\rho_2$.
  To establish this, we consider the following elements of $\Lambda$:
  \[ \tilde{\alpha}_1 = 1 + x^4 + x^8, \tilde{\alpha}_2 = 1 + x^4, \tilde{\alpha}_3 = 1 - x^{25}, \tilde\alpha_4 = 1 - x^{19}, \tilde\alpha_5 = 1 - x, \tilde{\alpha}_6 = x. \] 
  Then $c_2(\tilde{\alpha}_i) = \alpha_i$ for $1 \leq i \leq 6$, $\rho_2(\tilde{\alpha}_i) = 0$ for $i \in \{1,2,6\}$, and $\rho_2(\tilde{\alpha}_i) = 1$ for $i \in \{3,4,5\}$. 
  It follows that preimages of the $\beta_i$ under $c_2$ vanish under $\rho_2$. For example, $\tilde{\alpha}_1^{12} \tilde\alpha_3^6 \tilde\alpha_6^{27}$ is a preimage of $\beta_6$ under $c_2$ with
  image $\rho_2(\tilde \alpha_6) = 0$ under $\rho_2$.
\end{proof}

\begin{proof}[Proof of \cref{cor:N15}]
We have $Q_{56} = Q_{8 \cdot 7}$ and hence $p = 7$ with the notation of the preceding sections.
Let $u = u_{(N, 15)}$
and let $X = {(R_p/2R_p[C_2]^\times)}/({1+\mathfrak{p}R_pC_2})$.
Recall that by~\cref{prop:swansetup}~(iii), the maps $\rho_1$ and $\rho_2$ from~\cref{eq:rhos} induce a bijection
\[
\Z D_{4p}^\times \backslash \F_2 D_{4p}^\times /{\LL^\times}
\longrightarrow (\Z/2 \oplus X )/ (\rho_1,\rho_2)(\Lambda^\times).
\]
Now the composition of the canonical projections
$\Z/2 \oplus X \to
X
\to X/\rho_2(\Lambda^\times)
$ is trivial on $(\rho_1,\rho_2)(\Lambda^\times)$.
Thus we have an induced homomorphism
\[ \F_2 D_{4p}^\times \to \Z D_{4p}^\times \backslash \F_2 D_{4p}^\times /{\LL^\times } \to (\Z/2 \oplus X)/(\rho_1, \rho_2)(\Lambda^\times) \to X/\rho_2(\Lambda^\times),\]
which is in fact equal to $\rho_2$ by \cref{lemma:rho2units}.
Since $\rho_2(u_{28, 15}) \neq 0$ by \cref{lemma:rho2u}~(i), we conclude that $[u_{28, 15}]$ is not the trivial double coset and hence $(N, 15)$ not free by \cref{prop:swansetup}~(i).
\end{proof}

\section{Constructing CW-complexes from Swan modules} \label{s:CW}

The primary aim of this section will be to use the non-free stably free Swan module constructed in \cref{s:swan-modules} to prove \cref{thmx:CW}. In \cref{ss:CW-prelim}, we give preliminaries on the homotopy types of CW-complexes. In \cref{ss:CW-construction}, we give more details on \cref{construction:Swan} and its properties. We prove \cref{thmx:CW} in \cref{ss:CW-proof}, and then we prove \cref{thm:all-dimensions} in \cref{ss:CW-extra}.

\subsection*{Conventions} 
All sets whose elements are CW-complexes will be assumed to be sets of homotopy types. All spaces will be assumed to be connected and equipped with a choice of basepoint.

\subsection{Preliminaries on CW-complexes} \label{ss:CW-prelim}

For $n \ge 2$, a \textit{$(G,n)$-complex} is an $n$-dimensional CW-complex $X$ with fundamental group $G$ such that $\wt X$ is $(n-1)$-connected.

Let $\HT(G,n)$ denote the set of homotopy types of finite $(G,n)$-complexes. This has the structure of a graded tree with edges between each $X$ and $X \vee S^n$, and grading given by the directed Euler characteristic $\vv\chi (X) = (-1)^n \chi(X)$.
For fixed $G$ and $n$, this takes a minimal value $\chi_{\min}(G,n):= \min\{ \vv\chi(X) :  X \in \HT(G,n)\}$.
Let $\HT_{\min}(G,n) := \{ X \in \HT(G,n) : \vv\chi(X) = \chi_{\min}(G,n)\}$ denote the set of \textit{minimal} finite $(G,n)$-complexes.

It was shown by Dyer \cite[Theorem 1]{Dy78} that, if $G$ is a finite group and $n \ge 2$, then $\HT(G,n)$ has cancellation at level two, i.e.~for each $\ell \ge 2+\chi_{\min}(G,n)$, there is a unique $X \in \HT(G,n)$ with $\vv\chi(X) = \ell$. Thus, for a finite group $G$, $\HT(G,n)$ takes the general form given in \cref{figure:trees-example}, with some finite number (possibly equal to one) of vertices at levels $\ell = 0$ and $\ell=1$ respectively.

\begin{figure}[h] \vspace{-4mm} 
\begin{center}
\begin{tabular}{l}
\begin{tikzpicture}
\draw[fill=black] (-0.25,0) circle (2pt);
\draw[fill=black] (0.5,0) circle (2pt);
\draw[fill=black] (2,0) circle (2pt);
\draw[fill=black] (2,1) circle (2pt);
\draw[fill=black] (2,2) circle (2pt);
\draw[fill=black] (2,3) circle (2pt);
\draw[fill=black] (1.25,0) circle (2pt);
\draw[fill=black] (2.75,0) circle (2pt);
\draw[fill=black] (0.5,1) circle (2pt);
\draw[fill=black] (3.5,1) circle (2pt);
\draw[fill=black] (-1,1) circle (2pt);
\draw[fill=black] (5,1) circle (2pt);
\draw[fill=black] (-1,0) circle (2pt);
\draw[fill=black] (-1.75,0) circle (2pt);
\draw[fill=white] (5.85,0) circle (0pt);

\node at (2,3.8) {$\vdots$};
\draw[black] (2,1);
\draw[thick] (-0.25,0) -- (0.5,1)--(2,2) (0.5,0) -- (0.5,1) (2,0) -- (2,1) 
(2,1) -- (2,2) -- (2,3) (1.25,0)--(2,1)--(2.75,0) (3.5,1)--(2,2) (-1,1)--(2,2)--(5,1) (-1.75,0)--(-1,1)--(-1,0);
\end{tikzpicture}
\end{tabular}
\end{center}
\caption{General form of $\HT(G,n)$ when $G$ is a finite group}
\label{figure:trees-example}
\vspace{-2mm}
\end{figure}
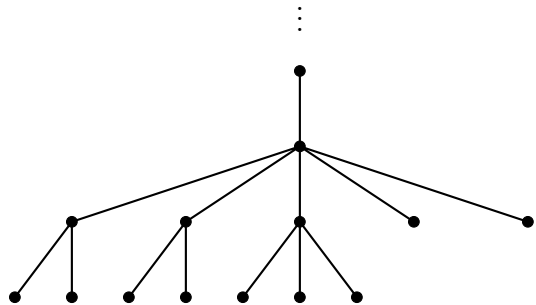

The following is a consequence of a theorem of Browning \cite[Theorem 5.4]{Br78} (see \cite[p252]{Dy79a}). Recall from the introduction that a finite group $G$ has \textit{free period $k$} if there exists a $k$-periodic resolution of finitely generated free $\Z G$-modules. 

\begin{lemma} \label{lemma:canc-except-periodic-case}
Let $n \ge 2$ and let $G$ be a finite group. If $G$ does not have free period $n+1$, then $\HT(G,n)$ has cancellation at level one, i.e.~for each $\ell \ge 1+\chi_{\min}(G,n)$, there is a unique $X \in \HT(G,n)$ with $\vv\chi(X) = \ell$.   
\end{lemma}

If $G$ is non-trivial and has free period $k=n+1$, then it has $k$-periodic cohomology and so $k$ is even (see, for example, \cite[Proposition 40.2]{Jo03a}) and $n$ is odd. It follows that, if $n \ge 2$ is even and $G$ is finite, then $\HT(G,n)$ has cancellation at level one (see also \cite[Corollary 4.7]{Ni20a}).

For $n \ge 2$ and an arbitrary finitely presented group $G$, a finite $(G,n)$-complex $X$ has an associated triple $T(X) = (\pi_1(X), \pi_n(X),k^{n+1}(X))$ where $G=\pi_1(X)$, where $\pi_n(X)$ is a $\Z G$-module (under the action coming from $\pi_n(X) \cong \pi_n(\wt X)$) and where $k^{n+1}(X) \in H^{n+1}(G;\pi_n(X))$ is the $(n+1)$th \textit{$k$-invariant} (see \cite{MW50,Dy76}). 
An abstract triple $T = (\pi_1,\pi_n,k)$ consists of a group $\pi_1$, a $\Z[\pi_1]$-module $\pi_n$ and a class $k \in H^{n+1}(\pi_1;\pi_n)$. We say two triples $T = (\pi_1,\pi_n,k)$ and $T' = (\pi_1',\pi_n',k')$ are isomorphic (written $T \cong T'$) if there exists a group isomorphism $\theta : \pi_1 \to \pi_1'$ and an isomorphism of $\Z[\pi_1]$-modules $f:\pi_n \to (\pi_n')_\theta$ such that $k \mapsto k'$ under the composition
\[ H^{n+1}(\pi_1;\pi_n) \xrightarrow[]{f_*} H^{n+1}(\pi_1;(\pi_n')_\theta) \xrightarrow[]{(\theta^{-1})^*} H^{n+1}(\pi_1';\pi_n'). \]
For more details, see \cite[p250]{Dy76}.

The following is standard and was shown by MacLane--Whitehead in the case $n=2$ \cite[Theorem 1]{MW50}, though the same argument works for all $n \ge 2$ (see \cite[p250]{Dy76}).

\begin{lemma} \label{lemma:triples}
Let $X$ and $Y$ be finite $n$-complexes with $(n-1)$-connected universal covers. Then $X \simeq Y$ if and only if $T(X) \cong T(Y)$.    
\end{lemma}

If $G$ is a finite group and $X$ is a finite $(G,n)$-complex, then dimension shifting gives an isomorphism $H^{n+1}(G;\pi_n(X)) \cong \wh H^0(G;\Z) \cong \Z/|G|$ which has the property that $k^{n+1}(X)$ maps to an element of $(\Z/|G|)^\times$ (see \cite[Section 2]{Dy76}). 
If $X$ and $Y$ are finite $(G,n)$-complexes for which there exists an isomorphism $\rho : \pi_n(X) \xrightarrow[]{\cong} \pi_n(Y)$ then we can choose an identification $i_Y : H^{n+1}(G;\pi_n(Y)) \xrightarrow[]{\cong} \Z/|G|$ and define $i_X = i_Y \circ H^{n+1}(G;\rho) : H^{n+1}(G;\pi_n(X)) \xrightarrow[]{\cong} \Z/|G|$.
This leads to classes $k^{n+1}(X)$, $k^{n+1}(Y) \in (\Z/|G|)^\times$ which are related by $k^{n+1}(X) = r \cdot k^{n+1}(Y)$ for some $r \in (\Z/|G|)^\times$ which is independent of the choice of $\rho$ and $i_X$ (see \cite[Section 2]{Dy76}). Thus $k$-invariants are often viewed as relative invariants (see \cite[Section 34]{Jo03a}).

\subsection{Construction and basic properties} \label{ss:CW-construction}

We begin by establishing the following, which proves that \cref{construction:Swan} is well-defined and has the properties mentioned in the introduction.

\begin{prop} \label{prop:construction-info}
Let $G$ be a finite group, let $r \in \SF(G)$, let $X$ be a finite $(G,n)$-complex and let $X_r$ be as defined in \cref{construction:Swan}. Then $X_r$ is a well-defined finite $(G,n)$-complex such that $\pi_n(X_r) \cong \pi_n(X)$ are isomorphic as $\Z G$-modules and $k^{n+1}(X_r) = r \cdot k^{n+1}(X) \in \Z/|G|$.
\end{prop}

\begin{proof}
First note that, in the notation of \cref{construction:Swan}, the chain complex $m^*(C_*(\wt X))$ is an exact sequence and has $\coker(\partial_1,0) \cong \Z$ as $\Z G$-modules by \cite[Lemma 4.12]{Ni20a}. The chain homotopy type of $m^*(C_*(\wt X))$ depends only on the homotopy type of $X$ (rather than the choice of cell structure) since pullbacks induce chain homotopy equivalences.
As described in \cref{construction:Swan}, we have that $m^*(C_*(\wt X)) \simeq C_*'$ for $C_*'$ an exact sequence of finitely generated free $\Z G$-modules. By exactness, we have $H_i(C_*') = 0$ for $1 \le i \le n-1$ and $H_0(C_*') \cong \Z$.

Thus \cite[Corollary 8.27]{Jo12a} implies there exists a finite $n$-complex $X_r$ with fundamental group $G$ such that $C_*(\wt X_r) \simeq C_*'$. More specifically, we $X_r$ can be constructed as follows:
\begin{clist}{(a)}
\item 
Add extensions of the form $\Z G \xrightarrow[]{\cong} \Z G$ to the $C_3$ and $C_2$ terms so that $m^*(C_{* \le 2}(\wt X)) \simeq C_*(\wt X_{\mathcal{P}})$ where $X_{\mathcal{P}}$ is the finite $2$-complex associated to some finite presentation $\mathcal{P}$ for $G$.
\item
Attach cells in dimensions $\ge 3$ to $X_{\mathcal{P}}$ to realise the remaining chains. This construction first appeared in work of Swan \cite[Lemma 3.1]{Sw60-II} and is attributed to Milnor.
\end{clist}

We have $H_i(\wt X_r) \cong H_i(m^*(C_*(\wt X)))=0$ for $1 \le i \le n-1$ and so $X_r$ is a finite $(G,n)$-complex. We have $\pi_n(X_r) \cong H_n(X_r) \cong H_n(m^*(C_*(\wt X))) = H_n(\wt X) \cong \pi_n(X)$ as $\Z G$-modules.
The relationship between the $k$-invariants follows from the formulation of the $k$-invariant given in \cite[Section 34]{Jo03a} and the existence of a chain map $f_* : m^*(C_*(\wt X)) \to C_*(\wt X)$ such that $f_0 : C_0(\wt X) \times_{\varepsilon,m} \Z \to C_0(\wt X)$ is the standard map $[(a,b)] \mapsto a$, and $f_i = \id$ for $i \ge 1$. This has the property that the induced map $(f_0)_* : \Z \cong H_0(m^*(C_*(\wt X))) \to H_0(\wt X) \cong \Z$ is $m$, i.e. multiplication by $r$. See \cite[Lemma 4.11]{Ni20a} and the diagram contained in its proof.
\end{proof}

The following gives a mild extension of a result of Dyer \cite[Theorem 9.1]{Dy76}, who showed that there exist bijections as in (i) and (ii), though without an explicit description of the map.

\begin{thm} \label{thm:main-homotopy-classification}
Let $G$ be a non-trivial finite group with free period $k$ and let $n=k-1$ be such that $n \ge 3$. Fix a minimal finite $(G,n)$-complex $X$. Then there are bijections:
\begin{clist}{(i)}
\item 
  $\SF(G)/{\pm{\Aut_{k}(G)}} \xrightarrow[]{\cong} \HT_{\min}(G,n)$, \, $r \mapsto X_r$.
\item 
$\SF(G)/ (\Aut_{k}(G) \cdot \FF(G)) \xrightarrow[]{\cong} \{Y \vee S^n : Y \in \HT_{\min}(G,n)\}$, \, $r \mapsto X_r \vee S^n$.
\end{clist}	
\end{thm}

In particular, in order to prove \cref{thmx:CW}, we need to find a finite group $G$ with minimal free period $k$ and an integer $i \ge 1$ such that $|\SF(G)/ (\Aut_{ik}(G) \cdot \FF(G))| > 1$.

We now give a direct proof of \cref{thm:main-homotopy-classification}. Our argument closely resembles that of \cite[Theorem 9.1]{Dy76} though the use of an explicit map $r \mapsto X_r$ (via \cref{construction:Swan}) leads to a streamlined proof.
For a $\Z G$-module $M$, let $\Aut_{\Z G}(M)$ denote the group of $\Z G$-module automorphisms of $M$. The proof of the following key lemma is elementary and involves determining $\Aut_{\Z G}(\Z \oplus \Z G)$ explicitly (see \cite[Proposition 4.1 (c)$\Rightarrow$(a)]{Dy76}).

\begin{lemma} \label{lemma:deg-aut}
Let $G$ be a finite group. Then the map
\[ \deg : \Aut_{\Z G}(\Z \oplus \Z G) \to \Aut(\wh H^0(G;\Z \oplus \Z G)) \cong (\Z/|G|)^\times  \]
has image $F(G) \le (\Z/|G|)^\times$.
\end{lemma}

\begin{proof}[Proof of \cref{thm:main-homotopy-classification}]
(i) First note that, since $G$ has free period $n+1$, it acts freely on a finite $n$-complex $K_0$ homotopy equivalent to $S^n$ \cite{Sw60-II}. Then $K=K_0/G$ is a finite $(G,n)$-complex with $\pi_n(K) \cong \Z$ as $\Z G$-modules. Since this has no $\Z G$-summands, $K$ is minimal. 

It follows that $\pi_n(Y) \cong \Z$ for any $Y \in \HT_{\min}(G,n)$, since we have $\pi_n(Y) \oplus \Z G^r \oplus \pi_n(K) \oplus \Z G^r$ for some $r \ge 0$, which implies that $\pi_n(Y) \cong \Z$. 
Thus, given $X \in \HT_{\min}(G,n)$, we obtain a map 
\[ \phi : \HT_{\min}(G,n) \to \SF(G)/{\pm{\Aut_{k}(G)}}, \quad Y \mapsto [k^{n+1}(Y) \cdot k^{n+1}(X)^{-1}].\] This is well-defined by the remarks at the end of \cref{ss:CW-prelim} as well as \cite[Lemma 7.3]{Sw60-II} which implies $k^{n+1}(X)$, $k^{n+1}(Y) \in \SF(G)$.
For injectivity, it follows from \cref{lemma:triples} that $X_r \simeq X_s$ if and only if  $(G,\Z,r) \cong (G,\Z,s)$, which holds if and only if there exists $(\theta,f) : (G,\Z) \to (G,\Z)$ such that $(\theta^{-1})^*(f^*(r))=s$. Since $\Z_\theta = \Z$, $f : \Z \to \Z$ is a $\Z G$-isomorphism and so $f = \pm \id$, and $(\theta^{-1})^* = \psi_{k}(\theta^{-1})$. Hence $s = (\theta^{-1})^*(f^*(r)) = r \cdot (\theta^{-1})^*(f^*(1)) = r \cdot (\pm \psi_{k}(\theta^{-1}))$ for some $\theta \in \Aut(G)$, i.e.~$sr^{-1} \in \pm{\Aut_{k}(G)}$.
It is surjective since, by \cref{prop:construction-info}, $\phi(X_r) = [k^{n+1}(X_r) \cdot k^{n+1}(X)^{-1}] = [r]$ for any $r \in \SF(G)$.
Thus $\phi$ is a bijection with inverse $r \mapsto X_r$.

(ii) It follows from (i) that $\{Y \vee S^n : Y \in \HT_{\min}(G,n)\} = \{X_r \vee S^n : r \in \SF(G)\}$, where both sets denote collections of homotopy types. By \cref{lemma:triples},  $X_r \vee S^n \simeq X_s \vee S^n$ if and only if $(G,\Z \oplus \Z G,r) \cong (G,\Z \oplus \Z G,s)$, which holds if and only if there exists $(\theta,f) : (G,\Z \oplus \Z G) \to (G,\Z \oplus \Z G)$ such that $(\theta^{-1})^*(f^*(r))=s$, i.e.~$(\theta^{-1})^*(f^*(1))=sr^{-1}$. Let $u : (\Z \oplus \Z G)_\theta = \Z \oplus \Z G_\theta \to \Z \oplus \Z G$ denote the standard isomorphism. Then $u \circ f \in \Aut_{\Z G}(\Z \oplus \Z G)$, and $(\theta^{-1})^* \circ u^{-1} = \psi_{k}(\theta^{-1})$ since this map fixes the coefficient module and hence acts only via a group automorphism. Hence $sr^{-1} = (\theta^{-1})^*(f^*(1)) = \psi_{k}(\theta^{-1}) \cdot \deg(u \circ f)$. Since $u \circ f \in \Aut_{\Z G}(\Z \oplus \Z G)$ and $\theta \in \Aut(G)$ are independent and can be arbitrary, \cref{lemma:deg-aut} implies this is equivalent to $sr^{-1} \in \Aut_{k}(G) \cdot \FF(G)$. The result follows.
\end{proof}

\subsection{Proof of \cref{thmx:CW}} \label{ss:CW-proof}

Let $n \ge 1$ with $n \equiv 3 \mod 4$. Let $G = Q_{56}$, which has minimal free period $4$. Then $n=4i-1$ for some $i \ge 1$. By \cref{thm:main-homotopy-classification} (ii), we have that 
\[ \{X \vee S^n : X \in \HT_{\min}(G,n)\} \cong \SF(G)/ (\Aut_{4i}(G) \cdot \FF(G)).\]
and so it suffices to prove that $\lvert \SF(G)/ (\Aut_{4i}(G) \cdot \FF(G)) \rvert > 1$.

By \cref{prop:SF(Q8p)}, we have that 
\[ \SF(G) = \{r \in (\Z/56)^\times : r \equiv \pm 1 \mod 8\} = \{ \pm 1, \pm 9, \pm 15, \pm 17, \pm 23, \pm 25\}.\]

By \cite[Proposition 1.1]{GG04}, we have
$\Aut_4(G) = ((\Z/56)^\times)^2 = \{1,9,25\}$. By \cref{lemma:psi_k-properties}, it
follows that $\Aut_{4i}(G) = (\Aut_4(G))^i = \{1,9,25\}^i$. Since $\{1,9,25\} \cong \Z/3$ as a group, we have $\{1,9,25\}^i \cong i \cdot \Z/3$ and so $\Aut_4(G)$ is equal to $\{1\}$ if $3 \mid i$ and $\{1,9,25\}$ otherwise.

We now complete the proof subject to the following lemma.

\begin{lemma} \label{cor:N9}
The Swan module $(N, 9)$ of $\Z[Q_{56}]$ is a free $\Z[Q_{56}]$-module.
\end{lemma}

Since $9^2 \equiv 25 \mod 56$, this implies that $(N,25)$ is free since $(N,25) \cong (N,9) \otimes (N,9)$ (i.e. since $\FF(G)$ is a subgroup of $(\Z/56)^\times$). Since $\Aut_{4i}(G) \subseteq \{1,9,25\}$ as shown above, we have that $\Aut_{4i}(G) \le \FF(G)$ and so $\SF(G)/(\Aut_{4i}(G) \cdot \FF(G)) = \SF(G) /{\FF(G)}$.

By \cref{cor:N15} (i.e.~\cref{thmx:SF-vs-F}), we have that $15 \not \in \FF(G)$ and so $\lvert\SF(G)/\FF(G)\rvert > 1$. In fact, since we always have $\{\pm 1\} \subseteq \FF(G)$ and we have $\{1,9,25\} \subseteq \FF(G)$ by the argument above, it follows that $\SF(G)/{\FF(G)} = \{[1],[15]\}$.
Hence $|\{X \vee S^n : X \in \HT_{\min}(G,n)\}| =2$, i.e. there exist finite $(G,n)$-complexes $X$ and $Y$ such that $X \vee S^n \not \simeq Y \vee S^n$. 
The result now follows from the fact that, since $G$ is finite, we have $X \vee 2S^n \simeq_s Y \vee 2S^n$ by \cite[Theorem 3]{Dy81}.

It remains to prove \cref{cor:N9}. The main ingredient in the proof will be \cref{lemma:rho2u} (ii).

\begin{proof}[Proof of \cref{cor:N9}]
We have $Q_{56} = Q_{8p}$ for $p = 7$. In the notation of \cref{ss:swan-free}, recall from \cref{prop:swansetup}~(iii) that the maps $\rho_1$ and $\rho_2$ from~\cref{eq:rhos} induce a bijection
\[
\overline{\rho} : \Z D_{4p}^\times \backslash \F_2 D_{4p}^\times /{\LL^\times}
\longrightarrow (\Z/2 \oplus X )/ (\rho_1,\rho_2)(\Lambda^\times)
\]
where $X = {(R_p/2R_p[C_2]^\times)}/({1+\mathfrak{p}R_pC_2})$, and $(N, 9)$ is free if and only if $\overline \rho([u_{(N, 9)}]) = 0$.

It follows from \cref{prop:swansetup}~(iv) and \cref{lemma:rho2u} (ii) that $\rho_1(u_{(N, 9)})=0$ and $\rho_2(u_{(N, 9)})=0$. Hence $\overline{\rho}([u_{(N, 9)}])=0$, which implies that $(N,9)$ is free by \cref{prop:swansetup}~(ii).
\end{proof}

Note that, whilst we did not need to explicitly describe $X$ and $Y$ in the proof, the choice of $X \in \HT_{\min}(G,n)$ is arbitrary. Since $n \equiv 3 \mod 4$, it is well-known that there is a free action of $G = Q_{56}$ on $S^n$ by diffeomorphisms (see, for example, \cite{Wa78}) and so we can take $X = S^n/G$. This is a finite $(G,n)$-complex since $\wt X \cong S^n$, and it is minimal since $\pi_n(X) \cong \pi_n(\wt X) \cong \Z$.

\subsection{CW-complexes in dimension $n \equiv 1 \mod 4$ after a single stabilisation} \label{ss:CW-extra}

In this section, we will prove the following result, which was mentioned in the introduction.

\begin{thm} \label{thm:all-dimensions}
Let $n \ge 1$ with $n \equiv 1 \mod 4$. If $X$, $Y$ are stably equivalent finite $n$-complexes with finite fundamental group and $(n-1)$-connected universal covers, then $X \vee S^n \simeq Y \vee S^n$.
\end{thm}

By results in \cref{ss:CW-prelim}, it suffices to prove that $|\{X \vee S^n : X \in \HT_{\min}(G,n)\}|=1$. By \cref{lemma:canc-except-periodic-case}, this always holds except possibly when $G$ has minimal free period $k$ dividing $n+1$. If this is the case then \cref{thm:main-homotopy-classification} (ii) implies that there is a bijection
\[ \{X \vee S^n : X \in \HT_{\min}(G,n)\} \cong \SF(G)/(\Aut_{n+1}(G) \cdot \FF(G)).\]
Since $n \equiv 1 \mod 4$, $k$ is even (see, for example, \cite[Proposition 40.2]{Jo03a}) and $k$ divides $n+1$, we must have that $k \equiv 2 \mod 4$. We will show:

\begin{lemma} \label{lemma:groups-eichler}
Let $k \ge 2$ with $k \equiv 2 \mod 4$. If $G$ is a finite group with $k$-periodic cohomology, then $G$ satisfies the Eichler condition.
\end{lemma}

By \cite{Sw62}, this implies that $\Z G$ has projective cancellation and so stably free Swan modules are free. Hence $\SF(G)=\FF(G)$ and so $|\{X \vee S^n : X \in \HT_{\min}(G,n)\}|=1$, which completes the proof of \cref{thm:all-dimensions}.

\begin{proof}[Proof of \cref{lemma:groups-eichler}]
By \cite[Theorem 1.10]{DM85}, we have $\per(G) = \lcm_{H \in \mathcal{H}}(\per(H))$ where $\per(\cdot)$ denotes the minimal cohomological period of a group with periodic cohomology and $\mathcal{H}$ denotes the set of subgroups $H \le G$ which are $p$-hyperelementary for some prime $p$. 

If $Q_{2^n} \le G$ for some $n \ge 3$ then, since $Q_{2^n}$ is $2$-hyperelementary, we have $4 = \per(Q_{2^n}) \mid \per(G)$. This is a contradiction since $\per(G) \equiv 2 \mod 4$.
Since $G$ has periodic cohomology, its Sylow subgroups are cyclic or quaternionic \cite[XII.11.6]{CE56}. Since $Q_{2^n}$ is not a subgroup of $G$, this implies that all the Sylow subgroups of $G$ are cyclic. It now follows from a theorem of Burnside (see \cite[Theorem 1.8]{DM85}) that $G$ is a split metacyclic group, i.e.~$G \cong C_m \rtimes_{(r)} C_n$ where $(n,m)=1$, $r \in \Z/m$ is such that $r^n \equiv 1 \mod m$ and $\rtimes_{(r)}$ denotes the semidirect product where the generator $y \in C_n$ acts on the generator $x \in  C_m$ via $x \mapsto x^r$.
By \cite[p165]{Jo03a}, we can assume that $m$ is odd.

Suppose for contradiction that $G$ fails the Eichler condition. Then $G \twoheadrightarrow H$ for $H$ a binary polyhedral group (see \cite[Proposition 1.5]{Ni21a}), and this induces a quotient of the Sylow $2$-subgroups $G_2 \twoheadrightarrow H_2$ (see \cite[Lemma 6.6]{Ni19}). Since $G_2$ is cyclic, $H_2$ must be cyclic and so $H \cong Q_{4a}$ for some $a \ge 3$ odd.
By \cite[Lemma 2.6]{Ni20b}, $G \twoheadrightarrow Q_{4a}$ if and only if $a \mid m$ and $r \equiv -1 \mod a$. This implies that $m=as$ for some $s \ge 1$ and, since $4 \mid |Q_{4a}| \mid |G|$ and $m$ is odd, we have $n=4t$ for some $t \ge 1$. Hence $G \cong C_{as} \rtimes_{(r)} C_{4t}$ where $a, s$ are odd and coprime to $n$, and $r \in \Z/as$ is such that $r^{4t} \equiv 1 \mod as$ and $r \equiv -1 \mod a$.

Let $\psi: \Z/4t \to (\Z/as)^\times$, $1 \mapsto r$ denote the action in the semidirect product. By \cite[p229]{DM85}, we have $\per(G) = 2 \cdot |\IM(\psi)|$. If $q : (\Z/as)^\times \twoheadrightarrow (\Z/a)^\times$ denotes reduction mod $a$, then $q(\psi(1))=-1$ and so $q(\IM(\psi)) = \{\pm 1\}$. This implies that $2 \mid |\IM(\psi)|$ and so $4 \mid \per(G)$, which is a contradiction.
\end{proof}

\section{Swan modules associated to group automorphisms}
\label{s:Swan-module-automorphisms}

In this section, we begin by recalling two facts (\cref{ss:Aut-prelim}) before determining $\IM(\psi_4)$ for the groups $Q(16,p,1)$ (\cref{ss:Aut-out}). We then use this in \cref{ss:Aut-proof} to prove \cref{thmx:automorphism} subject to calculations in \textsc{Magma} which are further explained in \cref{appendix:magma}.

\subsection{Preliminaries on the action of $\Aut(G)$} \label{ss:Aut-prelim}

Let $\Inn(G) \le \Aut(G)$ denote the group of inner automorphisms and let $\Out(G) = \Aut(G)/\Inn(G)$ denote the group of outer automorphisms. The following is an immediate consequence of the well-known fact that inner automorphisms induce the identity on group cohomology.

\begin{lemma} \label{lemma:out}
Let $G$ be a finite group with $k$-periodic cohomology. If $\theta \in \Inn(G)$, then $\psi_k(\theta) = 1$. In particular, $\psi_k$ factors as $\psi_k : \Out(G) \to (\Z/|G|)^\times$. 
\end{lemma}

The following was proven independently by Dyer \cite[Note (b) p276]{Dy76} and Davis \cite{Da83}.

\begin{prop} \label{prop:davis-dyer}
If $G$ has free period $k$, then $\Aut_k(G) \le \SF(G)$. That is, $(N,\psi_k(\theta))$ is stably free for all $\theta \in \Aut(G)$.    
\end{prop}

\subsection{Computing $\IM(\psi_4)$ for the groups $Q(16,p,1)$} \label{ss:Aut-out}

Let $p$ be an odd prime and define $Q(16,p,1) = C_p \rtimes_{(-1,1)} Q_{16}$ where $Q_{16} = \langle x,y \mid x^8=y^2, yxy^{-1}=x^{-1}\rangle$, $C_p = \langle z \mid z^p=1 \rangle$ and $xzx^{-1}=z^{-1}$, $yzy^{-1}=z$.
Let $i_{C_p} : C_p \hookrightarrow Q(16,p,1)$ and $i_{Q_{16}} : Q_{16} \hookrightarrow Q(16,p,1)$ denote the standard inclusion maps and let $q : Q(16,p,1) \twoheadrightarrow Q_{16}$ denote the standard quotient map. 
From now on, we will let $G=Q(16,p,1)$ and will write $Q_{16} =\langle x,y \rangle$ and $C_p = \langle z \rangle$ to refer to the subgroups of $G$ just defined.

The aim of this section will be to establish the following. Throughout, we will write $\psi_4^G$ to refer to the map $\psi_4$ in the case of the group $G$.

\begin{thm} \label{thm:image-psi-4}
Let $p$ be an odd prime and let $G = Q(16,p,1)$. Then $\IM(\psi_4^G) = ((\Z/16p)^\times)^2$. 
\end{thm}

Our strategy will be to relate $\psi_4^G$ to $\psi_4^{Q_{16}}$ and $\psi_4^{C_p}$.
First recall that, by \cref{lemma:out}, $\psi_4^G$ factors as $\psi_4^G : \Out(G) \to (\Z/|G|)^\times$ and similarly for $Q_{16}$ and $C_p$.

\begin{lemma} \label{lemma:Out(G)}
For $i \in (\Z/8)^\times$ and $j \in (\Z/p)^\times$, let $\theta_{i,j} : x \mapsto x^i, y \mapsto y, z \mapsto z^j$. This defines an automorphism $\theta_{i,j} \in \Aut(G)$. Then:
\[ \Out(G) = \{ [\theta_{i,j}] : i \in \{1,3\}, j \in (\Z/p)^\times\} \cong C_2 \times C_{p-1} \]
where $[\theta_{i,j}]$ denotes the image of $\theta_{i,j}$ in $\Out(G) = \Aut(G)/{\Inn(G)}$.
\end{lemma}

\begin{proof}
We will begin by showing that
    \begin{align*} &\Aut(G) = \{ \theta_{i,j,k,\ell} : x \mapsto z^kx^i, y \mapsto x^{2j}y, z \mapsto z^\ell \mid i \in (\Z/8)^\times, j \in \Z/4, k \in \Z/p, \ell \in (\Z/p)^\times \} \\
&= (\{ \theta_{1,0,k,1} : k \in \Z/p \} \rtimes \{ \theta_{1,0,0,\ell} : \ell \in (\Z/p)^\times \}) \times (\{ \theta_{1,j,0,1} : j \in \Z/4 \} \rtimes \{ \theta_{i,0,0,1} : i \in (\Z/8)^\times \}).
\end{align*} 
We can verify directly that $\theta_{i,j,k,\ell} \in \Aut(G)$ and that they form a subgroup as described.

First note $C_p =\langle z \rangle \le G$ contains all the elements of order $p$ in $G$ and so must be a characteristic subgroup. In particular, there is an induced map $\Psi: \Aut(G) \to \Aut(Q_{16})$, $\theta \mapsto \overline{\theta}$.
Let $\varphi \in \Ker(\Psi)$.
Since $\varphi(z)$ has order $p$, we must have $\varphi(z)=z^\ell$ for some $\ell$.
Since $\overline{\varphi(x)}=\overline{x} \in G/C_p$ and $\overline{\varphi(y)}=\overline{y} \in G/C_p$, we have that $\varphi(x)=z^kx$ and $\varphi(y)=z^ty$ for some $k, t \in \Z/p$. Since $y^4=1$, we have $(z^ty)^4 = \varphi(y^4)= 1$ but $(z^ty)^4=z^{4t}$ and so $z^{4t}=1$, which implies that $t=0$ since $z$ has order $p$. This implies that $\varphi = \theta_{1,0,k,\ell}$, and so $\ker(\Psi) = \{\theta_{1,0,k,\ell} \mid k \in \Z/p,\ell \in (\Z/p)^\times\}$.

By \cite[Proposition 8.2]{Sw60-II}, we have
$\Aut(Q_{16}) = \{ \alpha_{i,j} : x \mapsto x^i, y \mapsto x^jy \mid i \in (\Z/8)^\times, j \in \Z/8\}$.
Note that $\Psi(\{\theta_{i,j,0,1} : i \in (\Z/8)^\times, j \in \Z/4\}) = \{ \alpha_{i,2j} : i \in (\Z/8)^\times, j \in \Z/4\} \le \Aut(Q_{16})$.
We claim this is the image of $\Psi$. If not, then $\alpha_{i,j} \in \IM(\Psi)$ for some $j$ odd,which implies that $\alpha_{1,1} \in \IM(\Psi)$ by composition with automorphisms of the form $\theta_{i,j,0,1}$. Suppose $\varphi \in \Aut(G)$ is such that $\Psi(\varphi) = \alpha_{1,1}$. Then $\overline{\varphi(y)}=\overline{xy} \in G/C_p$ which implies that $\varphi(y)=z^txy$ for some $t \in \Z/p$ and, as usual, $\varphi(z)=z^\ell$ for some $\ell \in (\Z/p)^\times$. Since $yzy^{-1}=z$, we have $\varphi(y)\varphi(z)\varphi(y)^{-1}=\varphi(z)$ but $(z^txy)z^\ell(z^txy)^{-1} = z^{-\ell} \ne z^{\ell}$ since the fact that $z$ has order $p$ and $\ell \in (\Z/p)^\times$ implies that $z^{2\ell}\ne 1$. This is a contradiction and so $\IM(\Psi)$ is as required. 
Next note that the map $\alpha_{i,2j} \mapsto \theta_{i,j,0,1}$ gives a splitting of $\Psi$. Since the automorphisms $\theta_{i,j,0,1}$ and $\theta_{1,0,k,\ell}$ commute for all values of $i,j,k,\ell$, this implies that $\Aut(G)$ is a direct product of the subgroups $\{\theta_{1,0,k,\ell} : k \in \Z/p, \ell \in (\Z/p)^\times\}$ and $\{ \theta_{i,j,0,1} : i \in (\Z/8)^\times, j \in \Z/4\}$. The evaluation of $\Aut(G)$ now follows from the fact that $\theta_{i,j,k,\ell} = \theta_{i,j,0,1} \circ \theta_{1,0,k,\ell}$.

In order to determine $\Out(G)$, next note that
\[ \Inn(G) = \{ \theta_{i,j,k,1} \mid i \in \{\pm 1\}, j \in \Z/4, k \in \Z/p \} \cong C_p \times (C_4 \rtimes C_2). \]
This can be shown directly by noting that conjugation by a general element $x^iy^jz^k$ for some $i \in \Z/8$, $j \in \Z/2$, $k \in \Z/p$ gives an automorphism $\varphi : x \mapsto z^{2k(-1)^i}x^{(-1)^j}, y \mapsto x^{2i}y, z \mapsto z^{(-1)^i}$.
By quotienting, it now follows immediately that
$\Out(G) = \{ \overline{\theta}_{i,0,0,j} : i \in \{1,3\}, j \in (\Z/p)^\times\}$, where the automorphism $\theta_{i,0,0,j}$ coincides with $\theta_{i,j}$ in the statement.
\end{proof}

By \cite[Proposition 8.2]{Sw60-II}, we have that 
$\Aut(Q_{16}) = \{ \alpha_{i,j} : x \mapsto x^i, y \mapsto x^jy \mid i \in (\Z/8)^\times, j \in \Z/8\}$ and it is easy to show that $\Out(Q_{16}) = \{ [\alpha_{i,j}] : i \in \{1,3\}, j \in \{0, 1\}\}$. We have $\Out(C_p) = \Aut(C_p) = \{\beta_i : z \mapsto z^i \mid i \in (\Z/p)^\times\}$. With these choices of representatives, we can now define
\[ \Out(G) \to \Out(Q_{16}) \times \Out(C_p), \quad [\theta_{i,j}] \mapsto ([\theta_{i,j}\hspace{-.75mm}\mid_{Q_{16}}], [\theta_{i,j}\hspace{-.75mm}\mid_{C_p}]) = ([\alpha_{i,1}],[\beta_j]),\]
which is a group homomorphism.

\begin{lemma} \label{lemma:h4-split}
The standard inclusion maps induce an isomorphism 
\[ (i_{Q_{16}}^*,i_{C_p}^*) : H^4(G;\Z) \xrightarrow[]{\cong} H^4(Q_{16};\Z) \oplus H^4(C_p;\Z). \]
\end{lemma}

\begin{proof}
Since $|G| = 16p$, there is a primary decomposition $H^4(G;\Z) \cong H^4(G;\Z)_{(2)} \oplus H^4(G;\Z)_{(p)}$. The maps $i_{Q_{16}}^*$ and $i_{C_p}^*$ above factor through the induced maps $i_{Q_{16}}^* : H^4(G;\Z)_{(2)} \to H^4(Q_{16};\Z)$ and $i_{C_p}^* : H^4(G;\Z)_{(p)} \to H^4(C_p;\Z)$. By \cite[XII.10.1]{CE56}, these maps are injective since $Q_{16}$ is a Sylow $2$-subgroup and $C_p$ is a Sylow $p$-subgroup. It follows that $(i_{Q_{16}}^*,i_{C_p}^*)$ is injective.

Since $G$, $Q_{16}$ and $C_p$ all have 4-periodic cohomology, we have that $H^4(G;\Z) \cong \Z/16p$, $H^4(Q_{16};\Z) \cong \Z/16$ and $H^4(C_p;\Z) \cong \Z/p$ \cite[XII.11.1]{CE56}. So $(i_{Q_{16}}^*,i_{C_p}^*)$ is an injective map between finite groups of order $16p$, and hence is bijective.  
\end{proof}

Since $H^4(G;\Z) \cong \Z/|G|$ is cyclic, every subgroup is characteristic and so the surjective maps $i^*_{Q_{16}}$ and $i^*_{C_p}$ above induce maps on the automorphism groups. That is, there is an induced map
\[ (r_{Q_{16}},r_{C_p}) : \Aut(H^4(G;\Z)) \to \Aut(H^4(Q_{16};\Z)) \times \Aut(H^4(C_p;\Z)). \]

The following relates the two maps just defined with $\psi_4^G$, $\psi_4^{Q_{16}}$ and $\psi_4^{C_p}$.

\begin{lemma} \label{lemma:psi4-split}
There is a commutative diagram
\[
\begin{tikzcd}[column sep =20mm]
\Out(G) \ar[d,"\psi_4^G"] \ar[r,"\text{$[\theta_{i,j}] \mapsto ([\theta_{i,j}\hspace{-.75mm}\mid_{Q_{16}}],[\theta_{i,j}\hspace{-.75mm}\mid_{C_p}])$}"] & \Out(Q_{16}) \times \Out(C_p) \ar[d," \psi_4^{Q_{16}} \times \psi_4^{C_p}"] \\
\underbrace{\Aut(H^4(G;\Z))}_{\cong (\Z/16p)^\times} 
\ar[r,"\text{$(r_{Q_{16}},r_{C_p})$}"] & \underbrace{\Aut(H^4(Q_{16};\Z))}_{\cong (\Z/16)^\times} \times \underbrace{\Aut(H^4(C_p;\Z))}_{\cong (\Z/p)^\times}
\end{tikzcd}
\]
\end{lemma}

\begin{proof}
Let $[\theta_{i,j}] \in \Out(G)$ where $i \in \{1,3\}$, $j \in (\Z/p)^\times$. It suffices to check that $r_{Q_{16}}(\theta_{i,j}^*) = (\theta_{i,j}\hspace{-.75mm}\mid_{Q_{16}})^*$ and $r_{C_p}(\theta^*) = (\theta_{i,j}\hspace{-.75mm}\mid_{C_p})^*$.
For each $\varphi \in \Aut(H^4(G;\Z))$, the values of $r_{Q_{16}}$ and $r_{C_p}$ are uniquely determined by the fact that $r_{Q_{16}}(\varphi) \circ i_{Q_{16}}^* = i_{Q_{16}}^* \circ \varphi$ and $r_{C_p}(\varphi) \circ i_{C_p}^* = i_{C_p}^* \circ \varphi$ respectively.
For each subgroup $H = C_p$ or $Q_{16}$, we have $\theta_{i,j} \circ i_H = i_H \circ (\theta_{i,j}\hspace{-.75mm}\mid_{H})$. This implies that $(\theta_{i,j}\hspace{-.75mm}\mid_{H})^* \circ i_H^* = i_H^* \circ \theta_{i,j}^*$ for $H = C_p$ or $Q_{16}$, as required.
\end{proof}

\begin{proof}[Proof of \cref{thm:image-psi-4}]
Note that we can choose the identifications $\Aut(H^4(G;\Z)) \cong (\Z/16p)^\times$, $\Aut(H^4(Q_{16};\Z)) \cong (\Z/16)^\times$ and $\Aut(H^4(C_p;\Z))\cong (\Z/p)^\times$ such that $(r_{Q_{16}},r_{C_p})$ is the standard reduction isomorphism $r : (\Z/16p)^\times \xrightarrow[]{\cong} (\Z/16)^\times  \times (\Z/p)^\times$.
It now follows from \cref{lemma:psi4-split} that 
\[ r(\IM(\psi_4^G)) = \{ (\psi_4^{Q_{16}}(\theta_{i,j}\hspace{-.75mm}\mid_{Q_{16}}), \psi_4^{C_p}(\theta_{i,j}\hspace{-.75mm}\mid_{C_p})) : i \in \{1,3\}, j \in (\Z/p)^\times\} \le \IM(\psi_4^{Q_{16}}) \times \IM(\psi_4^{C_p}) \]
and we have that $\theta_{i,j}\hspace{-.75mm}\mid_{Q_{16}} = \alpha_{i,1}$ and $\theta_{i,j}\hspace{-.75mm}\mid_{C_p} = \beta_j$, in the notation defined above.

By \cite[Proposition 8.2]{Sw60-II}, we have that $\psi_4^{Q_{16}}(\alpha_{i,j}) = i^2$. This is well-defined mod $16$ since $(i+8)^2 = i^2 + 16i+64 \equiv i^2 \mod 16$. 
By \cite[Proposition 8.1]{Sw60-II}, we have that
$\psi_2^{C_p}(\beta_j) = j$  and so $\psi_4^{C_p}(\beta_j) = (\psi_2^{C_p}(\theta_j))^2 = j^2$ by \cref{lemma:psi_k-properties}.
Thus we have 
\[ r(\IM(\psi_4^G)) = \{(i^2,j^2) : i \in \{1,3\}, j \in (\Z/p)^\times\} =  ((\Z/16)^\times)^2 \times ((\Z/p)^\times)^2 \]
and so $\IM(\psi_4^G) = ((\Z/16p)^\times)^2$, as required.
\end{proof}

\subsection{Proof of \cref{thmx:automorphism}} \label{ss:Aut-proof}

Let $k \geq 1$ with $k \equiv 4 \bmod 8$, so that $k = 4i$ with $i \ge 1$ odd. 
Let $p = 5$ and let $G = Q(16,5,1)$.  By~\cref{thm:image-psi-4} we have
$\image(\psi_4) = ((\Z/80)^\times)^2 = \{1, 9, 41, 49\}.$
By \cref{lemma:psi_k-properties} we obtain
$\image(\psi_{4i}) = \image(\psi_4)^i = \image(\psi_4)$,
where the last equality follows from the fact that the order of $((\Z/80)^\times)^2 \cong \Z/2 \times \Z/2$ is coprime to $i$. It follows that, for all $k$, there exists $\theta \in \Aut(G)$ such that $\psi_k(\theta)=9$.

To finish the proof we will show that the Swan module $(N, 9)$ is not stably free.
To this end we consider the reduced projective class group $\wt K_0(\Z G)$ of $\Z G$.
For a finitely generated projective $\Z G$-module $X$, the class $[X] \in \wt K_0(\Z G)$ is trivial if and only if $X$ is stably free.
Thus the claim that the Swan module $(N, 9)$ is not stably free is equivalent to $[(N, 9)] \neq 0$.
Due to work of Bley--Boltje~\cite{MR2282916}, there exists an algorithm for determining the structure of $\tilde K_0(\Z G)$ as an abelian group, that is, $d_1,\dotsc,d_n \in \Z$ such that $\wt K_0(\Z G) \cong A$, where $A = \Z/d_1\Z \times \dotsb \times \Z/d_n\Z$. An explicit isomorphism $\wt K_0(\Z G) \to A$ is constructed in Bley--Wilson~\cite{MR2564571}, which can be used to decide whether a given finitely generated projective $\Z G$-module is stably free.
Both algorithms have been implemented in \textsc{Magma}, which can be used to show that
$\wt K_0(\Z G) \cong (\Z/2)^3 \times (\Z/4)$
and to check that $[(N, 9)] \neq 0$. See~\cref{appendix:magma} for details.
The computations have also been verified with an independent implementation using \textsc{Oscar}\cite{Hecke,OSCAR-book}.

\begin{remark}
We also used \cref{thm:image-psi-4} and \textsc{Magma} computations to analyse $\image(\psi_4^G)$ of $G=Q(16,p,1)$ for $p \leq 19$.
We show that, for $p \in \{ 3, 7, 11, 19 \}$, the Swan module $(N, \psi_4(\theta))$ is stably free for all $\theta \in \Aut(G)$.
For $p \in \{13, 17\}$, the groups $Q(16, p, 1)$ give further examples for \cref{thmx:automorphism} for all $k$ where in both cases one can take the Swan module $(N, 25)$. Note that unlike for $p \in \{5, 17\}$, in the case $p = 13$ one has $\image(\psi_4^G) \subsetneq \image(\psi_k^G)$ for all $k > 4$ with $k \equiv 4 \bmod 8$ and $k \equiv 0 \bmod 3$.
\end{remark}

\appendix

\section{Heuristic computational exploration} \label{appendix:heuristics}

The proof of \cref{thmx:SF-vs-F} shows that $Q_{56}$ does have not have the weak cancellation property, i.e. 
there exists a non-free stably free Swan module over $\Z Q_{56}$.
This group was identified as a potential candidate after extensive computational exploration, which we sketch in this section.

Let $G$ be a finite group.
Then $G$ has the weak cancellation property if and only if $\FF(G) = \SF(G)$.
Thus, to evaluate whether a finite group $G$ is a candidate for a group without the weak cancellation property, it suffices to determine $\FF(G)$ and $\SF(G)$.

The set
\[ \SF(G) = \ker(S_G) = \{ r \in (\Z/|G|)^\times : [(N,r)] = 0 \} \]
can be found algorithmically as
in \Cref{ss:Aut-proof}; see also \Cref{appendix:magma}.
This uses an algorithm of Bley--Wilson, which is quite efficient in practice and can easily handle groups up to order $100$.

For determining 
\[ \FF(G) = \{ r \in (\Z/|G|)^\times : (N,r) \cong \Z G\} \] 
the situation is more complicated.
Although by work of Bley, Johnston and the first named author~\cite{MR2422318,MR2813368,MR4136552}, there exist algorithms to decide whether a $\Z G$-module is free, for large group orders it quickly becomes impractical, rendering it useful only for groups of order $\leq 20$.
To overcome these limitations, we developed the following \textit{heuristic} version of the aforementioned algorithm.

\begin{customtheorem}{Algorithm}[Heuristic version of Bley--Johnston] \label{alg}
Let $\l = \Z G$ for a finite group $G$.
Suppose that $I \subseteq \l$ is a locally free $\l$-submodule of rank one.
Given $0 \leq \varepsilon \le 1$, the following steps either prove that $I$ is free or provide heuristic evidence that $I$ is non-free.
\begin{clist}{(1)}
\item
Compute a maximal order $\Gamma$ containing $\l$.
\item 
Determine whether $\Gamma I$ is a free $\Gamma$-module (see {\cite[\S{}5]{MR2813368}}).
If not, then $I$ is a non-free $\Lambda$-module and we terminate the algorithm.
\item 
Compute $\beta \in \Gamma$ such that $\Gamma I = \Gamma \beta$.
\item
Determine a finite set $S \subseteq \Gamma^\times$, such that $I$ is a free $\Lambda$-module if and only if there exists $u \in S$ such that $u\beta \in \Lambda$ (see {\cite[\S{}7]{MR2813368}}).
\item \label{item:heuristic}
Choose a random subset $S' \subseteq S$ of size $\lfloor \varepsilon \cdot |S| \rfloor$ and check whether $u \beta \not\in I$ for all $u \in S'$.
If this does not hold for some $u \in S'$, then $I$ is a non-free $\Lambda$-module and we terminate the algorithm. If this holds for all $u \in S'$, then we have heuristic evidence that $I$ is non-free.
\end{clist}
\end{customtheorem}

\begin{remark}
Note that the algorithm of Bley--Johnston coincides with the case where $\varepsilon = 1$, which corresponds to the case where $S = S'$ in step \cref{item:heuristic}. That is, if $u\beta \not \in I$ for all $u \in S$, then $I$ is non-free. 
This full algorithm is not feasible except in the case of groups of small order since heuristically the cardinality of $S$ grows super-exponentially with the group order.
\end{remark}

Applying this to the set of Swan modules $(N, r)$, $r \in (\Z/|G|)^\times$, we obtain a subset ${\FF'(G)} \subseteq \FF(G)$ which is dependent on the choice of $\varepsilon$ and $S'$. Since $\FF(G)$ is a subgroup of $(\Z/|G|)^\times$, we can additionally check that the computed subset $\FF'(G)$ is itself a group, i.e. that it is closed under multiplication. If so, this provides some heuristic evidence that $\FF'(G) = \FF(G)$.

We determined $\SF(G)$ and ${\FF'(G)}$ for quaternion groups $Q_{4n}$, where $2 \leq n \leq 19$, $n \not\in \{15,18\}$, and where we chose $\varepsilon$ in each case to be as large as possible provided the computations can finish in less than two months.
The computations were performed on a single CPU core of a server equipped with an Intel Xeon 6226R processor (2.90 GHz) and 768 GB of RAM.
The result is presented in \Cref{table:experiment}.
For $n \ne 14$, we obtained $\FF'(Q_{4n}) = \SF(Q_{4n})$. Since $\FF'(Q_{4n}) \subseteq \FF(Q_{4n}) \subseteq \SF(Q_{4n})$, this implies that $\FF'(Q_{4n}) = \FF(Q_{4n})$ and so the heuristic was accurate in all these cases.
For $n = 14$, we found that $\FF'(Q_{56}) \subsetneq \SF(Q_{56})$ and verified that $\FF'(Q_{56})$ is closed under multiplication. Since $15 \in \SF(Q_{56}) \setminus \FF'(Q_{56})$, this singled out $(N,15)$ as a potential candidate for a non-free stably free $\Z Q_{56}$-module.

\begin{table}[h]
\centering
\begin{tabular}{|c|P{11.5cm}|c|} 
\hline
$n$ & $\SF(Q_{4n})$ & $\FF'(Q_{4n})$ \\
\hline
$2$ & $\{1,7\}$ & $\SF(Q_{4n})$ \\
$3$ & $\{1,5, 7, 11\}$ & $\SF(Q_{4n})$ \\
$4$ & $\{1,7, 9, 15 \}$ & $\SF(Q_{4n})$ \\
$5$ & $\{1, 3, 7, 9, 11, 13, 17, 19\}$ & $\SF(Q_{4n})$ \\
$6$ & $\{1,23\}$ & $\SF(Q_{4n})$ \\
$7$ & $\{1, 3, 5, 9, 11, 13, 15, 17, 19, 23, 25, 27\}$ & $\SF(Q_{4n})$ \\
$8$ & $\{1, 7, 9, 15, 17, 23, 25, 31\}$ & $\SF(Q_{4n})$ \\
$9$ & $\{1, 5, 7, 11, 13, 17, 19, 23, 25, 29, 31, 35\}$ & $\SF(Q_{4n})$ \\
$10$ & $\{1, 9, 31, 39\}$ & $\SF(Q_{4n})$ \\
$11$ & $\{1, 3, 5, 7, 9, 13, 15, 17, 19, 21, 23, 25, 27, 29, 31, 35, 37, 39, 41, 43\}$ & $\SF(Q_{4n})$ \\
$12$ & $\{1, 23, 25, 47\}$ & $\SF(Q_{4n})$ \\
$13$ & $\{$1, 3, 5, 7, 9, 11, 15, 17, 19,
21, 23, 25, 27, 29, 31, 33, 35, 37, 41, 43, 45, 47, 49, 51$\}$
 & $\SF(Q_{4n})$ \\
$14$ & $\{1, 9, 15, 17, 23, 25, 31, 33, 39, 41, 47, 55\}$ & $\{1, 9, 25, 31, 47, 55\}$  \\
$16$ & $\{1, 7, 9, 15, 17, 23, 25, 31, 33, 39, 41, 47, 49, 55, 57, 63\}$ & $\SF(Q_{4n})$ \\
$17$ & $\{$1, 3, 5, 7, 9, 11, 13, 15, 19, 21, 23, 25, 27, 29, 31, 33, 35, 37, 39, 41, 43, 45, 47, 49, 53, 55, 57, 59, 61, 63, 65, 67$\}$ & $\SF(Q_{4n})$ \\
$19$ & $\{$1, 3, 5, 7, 9, 11, 13, 15, 17, 21, 23, 25, 27, 29, 31, 33, 35, 37, 39, 41, 43, 45, 47, 49, 51, 53, 55, 59, 61, 63, 65, 67, 69, 71, 73, 75$\}$ & $\SF(Q_{4n})$ \\
\hline
\end{tabular}
\vspace{2mm}
\caption{Stably free and free Swan modules for small quaternion groups.}
\vspace{-2mm}
\label{table:experiment}
\end{table}

For $n \in \{15,18\}$ or $n\geq 20$, steps (1)--(3) of the algorithm did not terminate within two months, and thus it was not possible to obtain any heuristic evidence for whether stably free Swan modules are free. 
Indeed, the runtime of the Bley--Johnston algorithm depends not only on the order of the group $G$, but also on the structure of the simple components of the group algebra $\Q G$.
However, the precise relationship has not yet been established.

While these computations do not constitute a proof due to the heuristic nature of the algorithm, the group $Q_{56}$ is identified as a candidate for a group without the weak cancellation property. In fact, the computations suggest that this might in fact be the smallest quaternion group with this property.

\section{\textsc{Magma} computations}\label{appendix:magma}

The following example was obtained using \textsc{Magma} version 2.25-6.
The necessary \textsc{Magma} files are available from
\begin{center}
    \url{https://www.mathematik.uni-muenchen.de/~bley/pub.php}
\end{center}
The following lines construct the class $X$ of the Swan module $(N, r)$ of $Q(16,p,1)$ and check whether it is non-zero.

\begin{verbatim}
Attach("RelAlgKTheory.m");
Attach("INB.m");
p := 5; r := 9;
Q := FPGroup< x, y |  y*x^p*y*x^-p, x*y*x*y^-3>; // Q(16,p,1)
G := RegularRepresentation(Q, sub<Q | >);
G := Codomain(G); // Q(16,p,1) as a permutation group
cl := LocallyFreeClassgroup(G);
QG := cl`QG;
rho := RegularRep(QG);
S := SwanModule(Integers()!r, rho);
X := ClassGroupLog(cl, S);
IsIdentity(X);
\end{verbatim}

\bibliography{stablyfree.bib}
\bibliographystyle{amsalpha}

\end{document}